\documentclass[10pt,reqno]{amsart}
\usepackage{amsmath,amssymb,bbold,hyperref}
\usepackage[arrow,matrix,curve]{xy}

\usepackage{mathrsfs,bbm,eucal}
\theoremstyle{plain}
\newtheorem{thm}{Theorem}[section]
\newtheorem*{thm*}{Theorem}
\theoremstyle{plain}
\newtheorem{lem}[thm]{Lemma}
\theoremstyle{plain}
\newtheorem{prop}[thm]{Proposition}
\theoremstyle{plain}
\newtheorem{example}[thm]{Example}
\newtheorem{cor}[thm]{Corollary}
\theoremstyle{definition}

\theoremstyle{remark}
\newtheorem{rem}[thm]{Remark}

\DeclareMathAlphabet{\mathpzc}{OT1}{pzc}{m}{it}

\newcommand{\dirac}{\mbox{$\mathcal{D}\!\!\!\!\!\:/\!\;$}}
\newcommand{\spinor}{\mbox{$S\!\!\!\!\!\:/\;\!$}}
\newcommand{\bundle}[1]{\CMcal{#1}}

\newcommand{\R}{\mathbbm{R}}
\newcommand{\C}{\mathbbm{C}}

\newcommand{\Q}{\mathbbm{Q}}
\newcommand{\F}{\mathbbm{F}}
\renewcommand{\H}{\mathscr{H}}
\newcommand{\Z}{\mathbbm{Z}}
\newcommand{\ke}{\mathpzc{k}}
\newcommand{\Di}{D}
\newcommand{\mlambda}{\lambda}

\begin{document}
\title{Relative K--area homology and applications}
\author{Mario Listing}

\begin{abstract}
We show an Uhlenbeck type estimate for closed simply connected manifolds which provides the existence of certain exact sequences in  K--area homology. This leads to the behavior of the K--area homology under surgery. Moreover, we give an index theoretic obstruction to positive scalar curvature on compact spin manifolds with boundary.
\end{abstract}
\keywords{Uhlenbeck estimate, K--area homology, surgery, scalar curvature}
\subjclass[2010]{55N35 (57R19, 57R20, 53C23)}
\maketitle

\section{Introduction}

We introduced in \cite{List10,habil} the notion of K--area homology  on the category  pairs of compact smooth manifolds and continuous maps, and proved that $\H _*$ is a functor on this category which satisfies the Eilenberg--Steenrod axioms up to the existence of long exact homology sequences. The following theorem proves the existence of \emph{very short} exact homology sequences if the submanifold has a finite fundamental group. In order to simplify notations we state the results in the introduction only for integral coefficients. However, the following theorem holds for arbitrary coefficients if the path components of $N$ are simply connected.
\begin{thm}
\label{main_thm}
Let $M^n$ be a compact manifold and $N\stackrel{i}\hookrightarrow M$ be a closed submanifold with finite fundamental group on all components, then the long sequence
\[
\stackrel{\partial _*}\longrightarrow \H _k(N )\stackrel{i_*}\longrightarrow\H _k(M  )\stackrel{j_*}\longrightarrow \H _k(M,N )\stackrel{\partial _*}\longrightarrow \H _{k-1}(N )\stackrel{i_*}\longrightarrow 
\]
is well defined with $\mathrm{Im}(i_*)=\ker (j_*)$, $\mathrm{Im}(j_*)= \ker (\partial _*)$ and $\mathrm{Im}(\partial _*)\subseteq \ker (i_*)$. 
\end{thm}

This theorem is more useful than it seems at first glance. In particular, it provides the behavior of the K--area homology under surgery. The assumption $| \pi _1(N,\, .\, )|<\infty $ is essential, counterexamples will be given in section \ref{sect4}. As usual the homomorphism $j_*$ is induced from $j:M\hookrightarrow (M,N)$ and $\partial _*$ is the connecting homomorphism from singular homology restricted to $\H _k(M,N)\subseteq H_k(M,N)$. The proof of the following corollaries are standard exercises in algebraic topology and will be left to the reader.

\begin{cor}
If $M'$ is obtained from $M^n$ by surgery in dimension $m\notin \{ 1,n-2\} $, then $\H _k(M')=\H _k(M)$ holds for all $k\not \in \{ m,m+1,n-m-1,n-m\} $. 
\end{cor}
We assume that surgery is done in the interior of $M$ respectively $M'$. If $M'$ is obtained from $M$ by a surgery along $S^m\subset M$ and $M$ is obtained from $M'$ by a surgery along $S^{n-m-1}\subset M'$, then excision shows 
\[
\H _k(M,S^m\times D^{n-m})=\H _k(M',D^{m+1}\times S^{n-m-1})
\]
for all $m$ and $k$. This implies $\H _k(M,S^m)=\H _k(M',S^{n-m-1})$ and hence, the last corollary follows immediately from the theorem. If $k\in \{ m,n-m-1\}$, then theorem \ref{main_thm} yields the relation between $\H _k(M)$ and $\H _k(M')$ if the maps $H_k(S^m)\to H_k(M)$ and $H_k(S^{n-m-1})\to H_k(M')$ are known explicitly. However, theorem \ref{main_thm} has its limits, because in general it is a nontrivial question to relate $\H _k(M')$ and $\H _k(M)$ if $k\in \{ m+1,n-m\} $. For instance, if $M^n_1$ and $M^n_2$ are connected orientable closed manifolds with $\H _n(M_1\# M_2)=0$, then theorem \ref{main_thm} can only show $\H _n(M_1\# M_2,S^{n-1})\neq \Z \oplus \Z $ for the connecting sphere $S^{n-1}\subset M_1\# M_2$, indeed $\H _4(T^4\# T^4,S^3)=0$ whereas $\H _{4}(T^4\# \C P^2,S^3)=\Z $.  The condition $m\notin \{ 1,n-2\}$ is even more essential because it is a highly nontrivial question to determine $\H _k(M)$ from $\H _k(M,S^1)$. 

\begin{cor}(Surgery in dimension $0$ and $n-1$)\\
Let $M^n$, $n\geq 3$, be a closed orientable manifold and $M'$ be obtained from $M$ by a $0$--dimensional surgery.
\begin{enumerate}
\item[a)] If $M$ and $M'$ have the same number of components, then $\H _k(M')=\H _k(M)$ holds for all $k\neq n-1$ and $\H _{n-1}(M')=\H _{n-1}(M)\oplus \Z $. 
\item[b)] If $M=M_0\coprod M_1$ and $M'=M_0\# M_1$ holds for connected manifolds $M_0$ and $M_1$, then $\H _k(M')=\H _k(M)$ is satisfied for all $k<n$ and $\H _n(M')=\Z $ holds if and only if $\H _n(M)=\Z \oplus \Z $.
\end{enumerate}

\end{cor}
\begin{cor}
If $M^n$ is closed and orientable with $\H _n(M)=H_n(M)$ and $M'$ is obtained from $M$ by a surgery in dimension $m\neq n-2$, then $\H _n(M')=H_n(M')$.
\end{cor}
The proof of this last corollary needs a case by case analysis. It shows in particular that finite K--area in Gromov's sense is preserved under surgery in dimension $m\neq n-2$. This was also proved in \cite{Fuku} if $n$ is even. The statement in the corollary fails obviously for surgeries in dimension $n-2$ as the example $M=S^2$ and $M'=T^2$ shows. Remarkable is the relation to questions of positive scalar curvature. We observed in \cite{List10} that a closed Riemannian spin manifold $(M^n,g)$ of positive scalar curvature satisfies $\H _n(M)=H_n(M)$. Moreover, if $M'$ is obtained from $M$ by a surgery in dimension $m\leq n-3$, then $M'$ admits a metric of positive scalar curvature. This was proved by Gromov, Lawson in \cite{GrLa1} respectively by Schoen, Yau in \cite{SchY6}. We notice that in general the existence of a positive scalar curvature metric is not preserved under surgery in dimension $m\in \{ n-2,n-1\} $. Theorem \ref{main_thm} may also be used to consider the behavior of the K--area homology under connected sums along closed submanifolds with finite fundamental group. The question of preserving positive scalar curvature in this context was treated by Ammann et.~al in \cite{pre_AmDH}. Using the APS index theorem we give an index theoretic obstruction to positive scalar curvature on compact spin manifolds with boundary which generalizes the result in \cite{APS2}:
\begin{thm}
\label{thm_aps_scalar}
Let $(M^n,g)$ be a compact Riemannian spin manifold of positive scalar curvature such that $g$ is a Riemannian product near $\partial M$, then
\begin{enumerate}
\item[(i)] $\int _M\widehat{A}(TM,\nabla ^{TM}) =\frac{\eta (\dirac _{\partial M})}2$ where $\eta (\dirac _{\partial M})$ means the $\eta $--invariant of the associated spin Dirac operator on $\partial M$.
\item[(ii)] The $\widehat{A}$--class of the tangent bundle satisfies
\[
\widehat{A}_k(TM)\cap [M]\in \H _{n-4k}(M,\partial M;\Q )
\]
for all $4k<n$ where $\cap [M]:H^{4k}(M;\Q )\to H_{n-4k}(M,\partial M;\Q )$ is the Poincar\'e--Lefschetz duality map. In fact, $\widehat{A}_0=1\in H^0(M;\Q )$ implies $\H _n(M,\partial M)=H_n(M,\partial M)$.
\end{enumerate}
\end{thm}

\section{Bundles of small curvature if $\pi _1(M)=0$}
In this section we extend methods and results \'a la Uhlenbeck \cite{Uhl1,Uhl2} and Gromov \cite{Gr01}. The main tool proving theorem \ref{main_thm} is theorem \ref{thm_uhl} below. Particularly important is the independence of the constants $C$ and $\epsilon $ from $\mathrm{rk}(\bundle{E})$, because in the presence of infinite K--area, the rank of the interesting bundles tends to infinity. Note that the assumption $\pi _1(M)=0$ in theorem \ref{thm_uhl} is essential. For instance, there exist nontrivial flat complex bundles on $\R P^n$ if $n\geq 2$. Moreover, even if a vector bundle is trivial in case $\pi _1(M)\neq 0$, in general we can not find a gauge transformation $s:M\to \mathrm{U}(m)$ which provides control of the connection form by the curvature like in theorem \ref{thm_uhl}.  For instance, a Hermitian connection $\nabla =d+A$ on the complex line bundle $S^1\times \C \to S^1$ is always flat: $R^\nabla =0$, however $\nabla $ is gauge equivalent to $d$ if and only if $\int _{S^1}A\in 2\pi \mathbf{i}\Z $. Observe that \cite[theorem 2.5]{Uhl2} is false for $n=2$, in the proof of \cite[lemma 2.2]{Uhl2} one needs $S^{n-2}$ to be connected. We also notice that the proof of \cite[theorem 1.3]{Uhl1} can not as easily be generalized to closed manifolds because an important part of this proof uses the contractibility of the base manifold (cf.~\cite[proof of lemma 2.3]{Uhl1}). Furthermore, the proof of \cite[section $4\frac{1}{4}$]{Gr01} does not provide a smooth unitary trivialization, i.e.~trying to extend Gromov's ansatz needs a more subtle selection theorem. 
\begin{thm}
\label{thm_uhl}
Let $(M^n,g)$ be a simply connected closed Riemannian manifold, then there are constants $\epsilon =\epsilon (M,g)>0$ and $C=C(M,g)>0$ with the following property: If $(\bundle{E},\nabla )\stackrel{\pi }\to M$ is a Hermitian vector bundle with curvature $\| R^\nabla \| _{g}<\epsilon $, then there is a smooth unitary trivialization $\Psi :\bundle{E}\to M\times \C ^m$ such that
\[
\| \nabla -\Psi ^*d\| _{g}\leq C\cdot \| R^\nabla \| _{g},
\] 
here $\Psi ^*d=\Psi ^{-1}d\Psi $ is the pullback of the canonical trivial connection $d$.
\end{thm}
Here and in the following $\| .\| _g$ denotes the $L^\infty $--comass operator norm on sections of $\Lambda ^*M\otimes \mathrm{End}(\bundle{E})$, in fact
\[
\| A\| _g=\sup _{x\in M}\{ |A(v_1,\ldots ,v_s)|_{op}\ |\ 0\leq s,v_i\in T_xM,|v_i|_g\leq 1\} 
\]
where $|.|_{op}$ denotes the operator norm on the fibers of $\bundle{E}$ induced by the Hermitian metric on $\bundle{E}$. If $A $ has values in $\mathrm{End}(\bundle{E})$, we omit the subscript $g$ and set $\| A\| =\sup _{x\in M}|A_x|_{op}$.
\begin{lem}[{\cite[$\text{sec. }4\frac{1}{4} \text{ eq. }(\Box )$]{Gr01}}]
\label{lemma_area}
Let $D\hookrightarrow (M,g)$ be a compact contractible surface with piecewise smooth boundary $\partial D$. If $(\bundle{E},\nabla )\to M$ is a Hermitian vector bundle, then the parallel transport $P$ on $\bundle{E}$ along $\partial D$ satisfies
\[
\| P-\mathrm{Id}\| \leq \mathrm{area}(D,g)\cdot \| R^\nabla _{|D}\| _g,
\] 
here $R^\nabla _{|D}$ denotes the curvature of $(\bundle{E},\nabla )_{|D}$.
\end{lem}
\begin{proof}
Fix a point $p\in \partial D$ and let $c:[0,1]\times [0,1]\to D\subseteq M$ be a piecewise smooth homotopy which parametrizes $D$ and contracts $D$ to $p$, i.e.~$c(s,0)=c(s,1)=p$ for all $s$, $c(0,t)=p$ for all $t$ and $t\mapsto c(1,t)$ is the loop based at $p$ which parametrizes $\partial D$. Let $P_s\in \mathrm{U}(\bundle{E}_p)$ be the parallel transport on $\bundle{E}$ along the closed loop $t\mapsto c_s(t):=c(s,t)$ and define the curvature transport
\[
R_{s,t}=P_{s,t}R(\partial _tc(s,t),\partial _sc(s,t))P_{s,t}^{-1}\in \frak{u}(\bundle{E}_p)
\]
where $P_{s,t}$ is the parallel transport along $c_s$ from $c_s(t)$ to $p=c_s(0)$.
Then $P_s$ is given by the product integral (cf.~\cite{Nij}):
\[
P_s=\prod _{[0,s]\times [0,1]}\left( \mathrm{Id}+R_{u,t}dtdu\right) =\prod _{0}^s\left( \mathrm{Id}+\left( \int _0^1R_{u,t}dt\right) du\right) .
\]
Thus, the parallel transport satisfies the differential equation $\partial _sP_s=\left( \int _0^1R_{s,t}dt\right) P_s$ respectively the integral equation
\[
P_s=\mathrm{Id}+\int _0^s\int _0^1R_{u,t}dtP_udu
\]
for all $s\in [0,1]$. Moreover, $P_{s,t}$ and $P_{s}$ are unitary operators, i.e.~$|R_{s,t}|_{op}\leq \| R^\nabla _{|D}\| _g\cdot |\partial _tc(s,t)\wedge \partial _sc(s,t)|_g$ yields the claim
\[
\| P_1-\mathrm{Id}\| \leq \| R^\nabla _{|D}\| _{g}\cdot \int _0^1\int _0^1|\partial _tc(s,t)\wedge \partial _sc(s,t)|_g dtds= \| R^\nabla _{|D}\| _{g}\cdot \mathrm{area}(D,g),
\]
note that $P_1$ is the parallel transport based at $p $ along $\partial D$. In the abelian case $\mathrm{rk}(\bundle{E})=1$, $R_{s,t}= R(\partial _tc(s,t),\partial _sc(s,t))$ yields the differential equation
\[
\partial _s\log P_s=\frac{\partial _sP_s}{P_s}=\int _0^1R(\partial _tc(s,t),\partial _sc(s,t))dt.
\]
Hence, if $D_s$ is the surface bounded by $c_s$ with $\| R^\nabla _{|D}\| _g\cdot \mathrm{area}(D_s)<\pi $, we conclude $P_s=\exp \left( \int _0^s\int _0^1R_{u,t}dtdu\right) $ which proves equation ($\Box $) in \cite[$4\frac{1}{4}$]{Gr01} for $\mathrm{rk}(\bundle{E})=1$:
\[
\| P_s-\mathrm{Id}\| \leq 2\sin \left( \frac{1}{2}\| R^\nabla _{|D}\| _g\cdot \mathrm{area}(D_s,g) \right) .
\]
In order to show this inequality in the nonabelian case one may use the above product integral. However, since the methods below can not provide optimal constants $\epsilon $, $C$ and we are only interested in the case $\| R^\nabla \| _g\cdot \mathrm{area}(D,g)\ll 1$ we omit the proof of Gromov's version which involves the sinus function. 
\end{proof}
\begin{prop}
\label{local_uhl}
Let $(M,g)$ be a Riemannian manifold, $\varrho _x$ be the injectivity radius at $x\in M$ and $B_r(x)\subseteq M$ be the open ball around $x$ of radius $r< \varrho _x$, then there is a constant $C=C(g,r)<+\infty $ with the following property. If $\pi :(\bundle{E},\nabla )\to M$ is a Hermitian vector bundle, then the parallel transport $P^\gamma _{y,x}:\bundle{E}_x\to \bundle{E}_y$ induced by $\nabla $ along the unique minimal geodesic $\gamma $ from $x$ to $y$ yields a smooth unitary trivialization $\Psi :\bundle{E}_{|B_r(x)}\to B_r(x)\times \bundle{E}_x$ with
\[
\| \nabla -\Psi ^*d\| _{g}\leq C\cdot \| R^\nabla \| _{g}
\]
on $B_r(x)$. Moreover, $C(r)=\frac{r}{2}$ for the Euclidean space $(M,g)=\R ^n$.
\end{prop}
\begin{proof}
If $y,z\in B_r(x)$ are sufficiently close, then the minimal geodesic $\kappa $ from $y$ to $z $ is contained in $B_r(x)\subseteq B_{\rho _x}(x)$. Let $\gamma $ be the minimal geodesic from $x$ to $y$ and $\tilde \gamma $ be the minimal geodesic from $x$ to $z$, then the geodesic triangle determined by $\gamma $, $\tilde \gamma $ and $\kappa $ is contained in $B_r(x)$. Consider the minimal compact surface $D_{x,y,z}\subseteq B_r(x)$ bounded by the loop $\sigma =\tilde \gamma \circ (-\gamma )\circ (-\kappa )$  and define
\[
C_y:=\limsup _{z\to y}\frac{\mathrm{area}(D_{x,y,z},g)}{\mathrm{dist}(y,z)}
\]
(in case $x$, $y$ and $z$ are geodesically collinear, set $\mathrm{area}(D_{x,y,z})=0$). Here $-\gamma $ respectively $-\kappa $ denote the path in the opposite direction, in fact $\sigma $ is the loop starting at $z$ going through $y$, then through $x$ and returning to $z$ by $\tilde \gamma $. We denote by $P^\gamma _{y,x}:\bundle{E}_x\to \bundle{E}_y$ the parallel transport along $\gamma $ from $x$ to $y$. Since $\gamma $ is the unique minimal geodesic from $x$ to $y$, $P_{y,x}^\gamma $ provides the unitary trivialization $\Psi :\bundle{E}\ni v\mapsto (\pi (v),(P_{\pi (v),x}^\gamma )^{-1}v)\in B_r(x)\times \bundle{E}_x$ (identify $\bundle{E}_x$ with $\C ^m$ by an unitary isomorphism). Note that $(P_{y,x}^\gamma )^{-1}=P^{-\gamma }_{x,y}$. We want to estimate the connection $\nabla $ using the parallel transport near $y$. Parallel transport is a unitary operator and $P_{z,z}^\sigma =P^{\tilde \gamma }_{z,x}P^{-\gamma }_{x,y}P^{-\kappa }_{y,z}$, i.e.~the above lemma yields
\[
\| P^{-\kappa }_{y,z}-P^{\gamma \circ (-\tilde \gamma )}_{y,z}\| =\| P_{z,z}^\sigma -\mathrm{Id}\| \leq \| R^\nabla \| _g\cdot \mathrm{area}(D_{x,y,z},g)
\]
Let $V\in \Gamma (\bundle{E}_{|B_r(x)})$ be parallel with respect to $\Psi ^*d$ and $|V(x)|=1$, then $V(y)=P^\gamma _{y,x}V(x)$ for all $y\in B_r(x)$. Consider a vector $w\in T_yM$ with $|w|=1$  and let $\kappa :(-\epsilon ,\epsilon )\to M$ be the geodesic with $\kappa (0)=y$ and $\kappa '(0)=w$, then
\[
(\nabla _wV)(y)=\frac{d}{dt}_{|t=0}\left[ P^{-\kappa }_{y,\kappa (t)}V(\kappa (t))\right] =\lim _{t\to 0}\frac{P^{-\kappa }_{y,\kappa (t)}V(\kappa (t))-V(y)}{t}
\]
by definition of the parallel transport. Hence, $V(y)=P_{y,z}^{\gamma \circ (-\tilde \gamma )}V(z)$  for $z=\kappa (t)$ shows the estimate
\[
| (\nabla _wV)(y)|_g \leq \lim _{t\to 0}\frac{\mathrm{area}(D_{x,y,\kappa (t)})}{|t|}\| R^\nabla \| _{g}\leq C_y\cdot \| R^\nabla \| _{g}\ ,
\]
here we use $\mathrm{dist}(\kappa (t),y)=|t|$. This proves the claim
\[
| (\nabla -\Psi ^*d)V|_g\leq C(r)\cdot \| R^\nabla \| _{g}\cdot |V|_g 
\]
on $B_r(x)$ for all $V\in \Gamma (\bundle{E})$ and $C(r):=\sup _{y\in B_r(x)}C_y$. Note that $\nabla -\Psi ^*d$ is a section in $\Omega ^1M\otimes \frak{u}(\bundle{E})$, where $\frak{u}(\bundle{E})$ denotes the bundle of skew Hermitian endomorphisms $\bundle{E}\to \bundle{E}$. In order to obtain $C(r)<\infty $, we need $r<\rho _x$. In fact, $r<\rho _x$ provides a constant $L>0$ with $\frac{1}{L}g_{x}\leq \exp _x^*g\leq L\cdot g_{x}$ on $B_r(0)\subseteq (T_xM,g_{x})$. Hence, $C(r)<+\infty $ follows from the Euclidean picture. In general $C(r)\to \infty $ if $r\to \varrho _x$, unless for instance $M=\overline{B_\varrho (x)}$. If $(M,g)$ is the Euclidean space $\R ^n$, we obtain $C(r)=r/2$ as follows: Triangles in $B_r(0)$ with vertices $0$, $y$ and $z$ have area less or equal to $\frac{r}{2}\cdot \mathrm{dist}(y,z)$.
\end{proof} 
Let's consider the standard sphere $(M,g)=S^n$ of constant sectional curvature $K=1$, then $C(r)=\tan \frac{r}{2}$ for all $0<r<\pi $. In order to see this we use spherical trigonometry, in fact we need the Delambre equations for spherical triangles. It suffices to consider the infinitesimal picture. Let $A$ be the area of a spherical triangle given by angles $\alpha $, $\beta $, $\gamma $ and opposite sides $a$, $b$, $c$, then $A=\alpha +\beta +\gamma -\pi $. We assume $c$ respectively $\gamma $ close to $0$, then $a\sim b$ and $\alpha \sim \beta $ yield
\[
\sin (a)\cdot \frac{\gamma }{2}\sim \sin \frac{a+b}{2}\sin \frac{\gamma }{2}=\cos \frac{\alpha -\beta }{2}\sin \frac{c}{2}\sim \frac{c}{2},
\]
i.e.~$c\sim \sin (a)\cdot \gamma $ for $\gamma \to 0$. Another Delambre equation shows (note $\gamma \to 0$ means $A\to 0$ and thus $\alpha +\beta \to \pi $)
\[
\begin{split}
\cos (a)\cdot \frac{\gamma }{2}\sim \cos \frac{a+b}{2}\sin \frac{\gamma }{2}&=\cos \frac{\alpha +\beta }{2}\cos \frac{c}{2}\\
&=-\sin \left(  \frac{\alpha +\beta }{2}-\frac{\pi }{2}\right) \cos \frac{c}{2}\sim \frac{\pi -\alpha -\beta }{2}.
\end{split}
\]
Hence, we conclude for the area in case $\gamma \to 0$: $A\sim -\gamma \cos a+\gamma $ which proves $A\sim c\cdot \frac{1-\cos a}{\sin a}=c\cdot \tan \frac{a}{2}$ where $a$ is the distance from $y$ to $x$, the origin in $B_r(x)$, and $c$ is the distance from $y$ to $z$ in the above proof. Note that in case $r\to 0$, we obtain the Euclidean picture $C(r)=\frac{1-\cos r}{\sin r}=\tan \frac{r}{2}\sim \frac{r}{2}$ which confirms the optimality of the value $C(r)$. Moreover, if $r\to \rho =\pi $, then $C(r)\to \infty $. Before we can continue to show theorem \ref{thm_uhl} for the standard sphere, we have to notice the following. 
\begin{rem}
\label{rem_exp}
Since $2\sin \frac{|t|}{2}=|e^{it}-1|$, we obtain $2\sin \frac{\| A\| }{2}=\| e^A-\mathrm{Id}\| $ for all $A\in \frak{u}(m)$ with $\| A\| \leq \pi $ and thus
\[
\exp :\left\{ A\in \frak{u}(m)\ |\ \| A\| <\frac{1 }{2} \right\} \to \left\{ B\in \mathrm{U}(m)\ |\ \| B-\mathrm{Id}\| <2\sin \frac{1}{4}\right\} \ ,\ A\mapsto e^A
\]
is a diffeomorphism with
\[
\| d\exp _A-\mathrm{Id}\| \leq e^{\| A\| }-1<\frac{2}{3}\quad \text{ and }\quad \| (d\exp _A)^{-1}\| \leq \frac{1}{2-e^{\| A\| }}<3
\]
for $\| A\| <\frac{1}{2}$.
\end{rem}
Choose $r=\frac{2}{3}\pi $ for simplicity, a point $x_1\in S^n$ and set $x_2=-x_1$, then $C(r)=\sqrt{3}$. We identify $\bundle{E}_{x_1}$ with $\C ^m$ and fix a unitary isomorphism $P^\kappa _{x_1,x_2}:\bundle{E}_{x_2}\to \bundle{E}_{x_1}$ for a minimal geodesic $\kappa $ from $x_2$ to $x_1$. Let $U_i=B_r(x_i)$ and $\Psi _i:\bundle{E}_{|U_i}\to U_i\times \C ^m$ be the trivializations from the above proposition where $\Psi _1(v)=\left( \pi (v),P^{-\gamma _1}_{x_1,\pi (v)}v\right) $ and where $\Psi _2(v)=\left( \pi (v),P^\kappa _{x_1,x_2}P^{-\gamma _2} _{x_2,\pi (v)}v\right) $ if $\gamma _i$ is the respective minimal geodesic from $x_i$ to $\pi (v)$. The coordinate transition function $g_{12}:U_1\cap U_2\to \mathrm{U}(m)$ is given by
\[
g_{12}(y)=\Psi _1\circ \Psi _2^{-1}(y)=P^{-\gamma _1}_{x_1,y}P_{y,x_2}^{\gamma _2}P^{-\kappa }_{x_2,x_1}=P^{(-\gamma _1)\circ \gamma _2\circ (-\kappa )}_{x_1,x_1}\in \mathrm{U}(\bundle{E}_{x_1}).
\]
The closed geodesic segment $(-\gamma _1)\circ \gamma _2\circ (-\kappa )$ bounds a minimal disc of area less or equal to $2\pi =\mathrm{Vol}(S^2)/2$, i.e.~lemma \ref{lemma_area} shows
\[
\left\| g_{12}-\mathrm{Id}\right\| =\left\| P_{x_1,x_1}^{(-\gamma _1)\circ \gamma _2\circ (-\kappa )}-\mathrm{Id}\right\| \leq 2\pi \cdot \| R^\nabla \| _{g},
\]
here we need the condition $\pi _1(S^n)=0$, i.e.~$n\geq 2$. In fact, this is the only point where the proof fails for $S^1$. Assuming  $\| R^\nabla \| _g< \frac{\sin \frac{1}{4}}{\pi }$ means $\| g_{12}-\mathrm{Id}\| < 2\sin \frac{1}{4}$, i.e.~the map
\[
\alpha _2=\exp (\chi _\varepsilon \cdot \exp ^{-1}g_{12}):U_2\to \mathrm{U}(\bundle{E}_{x_1})
\]
is well defined and smooth for a smooth cut off function $\chi _\varepsilon :U_2\to [0,1]$ with $\chi _\varepsilon (U_2\setminus U_1)=0$ and $\chi _\varepsilon (y)=1$ for all $y\in U_2$ with $\mathrm{dist}(y,\partial \overline{U_2})\leq \varepsilon $. We can choose $\chi _\varepsilon $ such that $| d\chi _\varepsilon | _g\leq \frac{1}{\pi /3-\varepsilon }+\varepsilon $.  Now $\| \exp ^{-1}g_{12}\| \leq 2\arcsin (\pi \| R^\nabla \| _g)$, $\| d\exp _\cdot  \| \leq 5/3$ and $\| d\exp ^{-1}g_{12}\| \leq 3\| dg_{12}\| $ yield
\[
\begin{split}
\| d\alpha _2\| _g&\leq \| d\exp _{\chi \exp ^{-1}g_{12}}\| \cdot (|d\chi _\varepsilon |_g\cdot \| \exp ^{-1}g_{12}\| +\| d\exp ^{-1}g_{12}\| _g)\\
&\leq \left( \frac{10}{\pi -3\varepsilon }+\frac{10}{3}\varepsilon \right) \arcsin (\pi \| R^\nabla \| _g)+5\| dg_{12} \| _g
\end{split}
\]
on $U_1\cap U_2$. The connection on $\bundle{E}$ provides an estimate of $dg_{12}$ as follows:
\[
\Psi _1^*(d+A_1)=\Psi _1^{-1}(d+A_1)\Psi _1=\nabla =\Psi _2^{-1}(d+A_2)\Psi _2=\Psi _2^*(d+A_2)
\]
shows $g_{21}dg_{12}+g_{21}A_1g_{12}=A_2$ on $U_1\cap U_2$. Proposition \ref{local_uhl} implies $\| A_i\| _{g}\leq \sqrt{3} \| R^\nabla \| _g$ and hence $\| dg_{12}\| _{g}\leq 2\sqrt{3}\| R^\nabla \| _{g}$ which proves $\| d\alpha _2\| _g\leq 20\sqrt{3}\| R^\nabla \| _g$ for sufficiently small $\varepsilon > 0$ (use $\arcsin (x)<\sqrt{3}\cdot x$ for $x<1/2$). Consider the trivialization $\Phi _1=(\Psi _1)_{|V_1}$ with
\[
V_1=U_1\setminus \overline{\{ y\in U_1\cap U_2|\chi _\epsilon (y)<1\} } 
\]
and $\Phi _2 :=\alpha _2\Psi _2$ on $V_2=U_2$. Then $\{ V_1, V_2\} $ covers $S^n$ and $\Phi _1\circ \Phi _2^{-1}=\Psi _1\circ \Psi _2^{-1}\alpha _2^{-1}=g_{12}\cdot \alpha _2^{-1}=\mathrm{Id}$ holds on $V_1\cap V_2=\mathrm{int} \{ y\in U_2\ |\ \chi _\varepsilon  (y)=1\}$.  Hence, $\Phi =\{ \Phi _1,\Phi _2\}:\bundle{E}\to S^n\times \C ^m$ is a trivialization of $\bundle{E}$ with
\[
\| \nabla -\Phi ^*d\| _{g}\leq \sqrt{3} \| R^\nabla \| _{g}
\]
on $S^n\setminus U_1\cap U_2$. Moreover, $\Phi _2=\alpha _2\Psi _2$ provides the needed estimate on $U_1\cap U_2$:
\[
\Phi _2^*(d+A)=\Phi _2^{-1}(d+A)\Phi _2=\nabla =\Psi _2^{-1}(d+A_2)\Psi _2=\Psi _2^*(d+A_2)
\]
implies $A_2=\alpha _2^{-1}d\alpha _2+\alpha _2^{-1}A\alpha _2$ and thus
\[
\| \nabla -\Phi ^*d\| _{g}=\| A\| _{g}\leq \| A_2\| _{g}+\| d\alpha _2\| _{g}\leq 21\sqrt{3} \| R^\nabla \| _{g}
\]
on $U_1\cap U_2$. We summarize these estimates in a corollary. The statement a) is an easy consequence of the results in \cite{List10,habil}. 
\begin{cor}
Let $(\bundle{E},\nabla )\to S^n$ be a Hermitian vector bundle on the standard sphere of constant sectional curvature $K=1$.
\begin{enumerate}
\item[a)] If $\| R^\nabla \| _{g}<\frac{1}{2}$, then $\bundle{E}$ is stably trivial, i.e.~$\bundle{E}\oplus \C ^k$ is trivial for some $k\geq 0$.
\item[b)] If $\| R^\nabla \| _{g}<\frac{1}{13 }$, then there is a unitary trivialization $\Phi :\bundle{E}\to S^n\times \C ^m$ with
\[
\| \nabla -\Phi ^*d\| _{g}\leq 21\sqrt{3} \| R^\nabla \| _{g}.
\]
\end{enumerate}
\end{cor}
The constant in a) is optimal for even $n$ whereas the constants in b) are certainly not optimal. In order to obtain optimal results one has to use a Coloumb gauge and not an exponential gauge (cf.~\cite{Uhl2}). Using equation $(\Box)$ in \cite[sec. $4\frac{1}{4}$]{Gr01} instead of lemma \ref{lemma_area}, a Hermitian bundle $(\bundle{E},\nabla ) \to S^n$ on the standard sphere is trivial if $\| R^\nabla \| _g< \frac{1}{6}$. It is unknown if the value $\frac{1}{6}$ provides an estimate of the connection form by the curvature up to gauge transformation like in theorem \ref{thm_uhl}. Another interesting question is the existence of nontrivial Hermitian  bundles $(\bundle{E},\nabla )\to S^{2n}$ with curvature $\frac{1}{6}\leq \| R^\nabla \| _g<\frac{1}{2}$.

Let's come back to the general case and prove theorem \ref{thm_uhl}. Assume that $(M,g)$ is a simply connected closed Riemannian manifold with injectivity radius $\varrho =\varrho (M)$. Let $\{ U_i=B_r(x_i)\} $ be a finite open cover of $M$ for $r<\varrho $ and define $C:=\max _iC_i(r)$ where the $C_i(r)$ are the constants from proposition \ref{local_uhl} for $B_r(x_i)$. We identify $\bundle{E}_{x_1}=\C ^m$ and fix minimal geodesics $\kappa _i$ from $x_i$ to $x_1$. Let $\Psi _i:\bundle{E}_{|U_i}\to U_i\times \C ^m$ be the trivializations from proposition \ref{local_uhl} where $\Psi _i(v)=(\pi (v),P_{x_1,x_i}^{\kappa _i}P_{x_i,\pi (v)}^{-\gamma _i}v)\in U_i\times \C ^m$ for the unique minimal geodesic $\gamma _i$ from $x_i$ to $\pi (v)\in B_r(x_i)$. For notational simplicity we also omit the coordinate part, i.e.~$\Psi _i(v)=P_{x_1,x_i}^{\kappa _i}P_{x_i,\pi (v)}^{-\gamma _i}v$. The coordinate transition functions $g_{ij}:U_i\cap U_j\to \mathrm{U}(\bundle{E}_{x_1})$ are given by
\[
g_{ij}(y)=\Psi _i\circ \Psi _j^{-1}(y)=P_{x_1,x_i}^{\kappa _i}P_{x_i,y}^{-\gamma _i}P_{y,x_j}^{\gamma _j}P_{x_j,x_1}^{-\kappa _j}=P_{x_1,x_1}^{\sigma _{ij}(y)}\in \mathrm{U}(\bundle{E}_{x_1})=\mathrm{U}(m)
\]
where $\sigma _{ij}(y)=\kappa _i\circ (-\gamma _i)\circ \gamma _j\circ (-\kappa _j)$ is the loop based at $x_1$ and through $y$. If $D_{ij}(y)$ is a disc which bounds the contractible loop $\sigma _{ij}(y)$ through $y$, then
\[
\| g_{ij}(y)-\mathrm{Id}\| =\Bigl\| P^{\sigma _{ij}(y)}_{x_1,x_1}-\mathrm{Id}\Bigl\| \leq \mathrm{area}(D_{ij}(y))\cdot \| R^\nabla \| _{g}
\]
holds for all $y$. Here we need $\pi _1(M)=0$, otherwise $\sigma _{ij}(y)$ may not be contractible. Let $D_{ij}^\mathrm{min}(y)$ be the minimal disc bounded by the loop $\sigma _{ij}(y)$ and define $\beta =\max _{i,j,y}\mathrm{area}(D_{ij}^\mathrm{min}(y))<+\infty $, then
\[
\| g_{ij}-\mathrm{Id}\| \leq \beta \cdot \| R^\nabla \| _{g}
\]
holds for all $i,j$. Hence, assuming $\| R^\nabla \| _g<\frac{2}{\beta }\sin \frac{1}{4}$ implies $\| g_{ij}-\mathrm{Id}\| <2\sin \frac{1}{4}$, i.e.~$\exp ^{-1}g_{ij}:U_i\cap U_j\to \frak{u}(m)$ is well defined with
\[
\|\exp ^{-1}g_{ij}\|= 2\arcsin \frac{\| g_{ij}-\mathrm{Id}\| }{2}\leq \frac{\pi\cdot \beta }{3} \| R^\nabla \| _{g}
\]
Now we can apply \cite[prop.~3.2]{Uhl1} to the coordinate transition functions $g_{ij}$ and $h_{ij}=\mathrm{Id}$ on $U_i\cap U_j$. In particular, there is an $\epsilon >0$ such that $\| R^\nabla \| _g<\epsilon $ implies the existence of a smaller cover $V_i\subseteq U_i$ of $M$ and smooth functions $\rho _i:V_i\to \mathrm{U}(m)$ with $h_{ij}=\rho _ig_{ij}\rho _j^{-1}$ on $V_i\cap V_j$. Define the trivializations $\Phi _i=\rho _{i}\Psi _i:\bundle{E}_{|V_i}\to V_i\times \C ^m$ then
\[
\Phi _i\circ \Phi _j^{-1}=\rho _i\Psi _i\Psi _j^{-1}\rho _{j}^{-1}=\rho _ig_{ij}\rho _j^{-1}=h_{ij}=\mathrm{Id}
\]
provides a global unitary trivialization $\Phi =\{ \Phi _i\} :\bundle{E}\to M\times \C ^m$. Thus, it remains to show an estimate of $\Phi ^*A:=\nabla -\Phi ^*d\in \Gamma (T^*M\otimes \mathrm{End}(\bundle{E}))$ depending on the curvature. By construction of the $\rho _i$ (cf.~\cite[prop.~3.2 and cor.~3.3]{Uhl1}) and remark \ref{rem_exp}, there is a constant $C'>0$ with
\[
\| d\rho _i\| _{g}\leq C'\max _{j,k,l,s,r}(|d\varphi _j|_g\cdot \| \exp ^{-1}g_{kl}\| +\| dg_{sr}\| _{g})\leq C''\cdot  \| R^\nabla \| _{g}
\]
on $V_i$ where $\{ \varphi _i \}$ is the partition of unity from the proof of \cite[prop.~3.2]{Uhl1} and $C'':=C'(2C+\frac{\pi \beta }{3}\max _j|d\varphi _j|_g)$. Here we use
\[
g_{ji}dg_{ij}+g_{ji}A_ig_{ij}=A_j
\]
on $U_i\cap U_j$ for the connection forms $A_i=(\Psi _i^{-1})^*\nabla -d$ and $\| A_i\| _{g}\leq C \| R^\nabla \| _g$ by the definition of the $\Psi _i$ and proposition \ref{local_uhl}, i.e.~$\|dg_{ij}\| _{g}\leq 2C \| R^\nabla \| _{g}$. Note that $C'$ depends only on the constants in remark \ref{rem_exp} and the number $\# \{ U_i\} $, hence $C''=C''(\epsilon )$ depends only on the choice of the cover $U_i$ and the subcover $V_i$, in fact $C''$ is independent of the bundle respectively the $g_{ij}$. We also notice that $C'$ is invariant under scaling of the Riemannian metric $g$ whereas $C$ and $C''$ scale like $C(t^2\cdot g)=t\cdot C(g)$ for constants $t>0$. $\Phi =\rho _i\Psi _i$ on $V_i$ implies
\[
\Psi _i^*(d+A_i)=\nabla =\Phi ^*(d+A)=\Psi _i^{-1}\rho ^{-1}_i(d+A)\rho _i\Psi _i=\Psi _i^*(d+\rho _i^{-1}d\rho _i+\rho _i^{-1}A\rho _i)
\]
which yields $A_i=\rho _i^{-1}d\rho _i+\rho _i^{-1}A\rho _i$ on $V_i$. Thus,
\[
\| \nabla -\Phi ^*d\| _g=\| A\| _g\leq \max _i(\| A_i\| _{g}+\| d\rho _i\| _{g})\leq (C+C'')\cdot \| R^\nabla \| _{g}
\]
completes the proof of theorem \ref{thm_uhl}. 

In order to show theorem \ref{main_thm} for odd $k$ we need a further generalization of theorem \ref{thm_uhl}. We introduce the following semi--norm on sections of $\Lambda ^k(M\times W)\otimes \mathrm{End}(\bundle{E})$:
\[
\| A\| _{TM,g}:=\sup _{(x,q)\in M\times W}\{ |A(v_1,\ldots ,v_k)|_{op}  \ |\ v_i\in T_xM\oplus \{ 0\} _q, |v_i|_g\leq 1\} ,
\]
here $|.|_{op}$ is the operator norm on $\bundle{E}_{(x,q)}$ induced by the Hermitian metric on $\bundle{E}$ and $g$ is a Riemannian metric on $M$ lifted to $\mathrm{pr}^*_M(TM)\subseteq T(M\times W)$.
\begin{thm}
\label{thm_uhl2}
Let $(M,g)$ and $C,\epsilon $ be as in theorem \ref{thm_uhl}, and $W$ be a manifold. If $(\bundle{E},\nabla )\stackrel{\pi }\to M\times W$ is a Hermitian vector bundle with curvature $\| R^\nabla \| _{TM,g}<\epsilon $ and $\bundle{E}_{x,W}:=\pi ^{-1}(x,W)$ is the induced bundle over $W$ for some $x\in M$, then there is a unitary bundle isomorphism $\Phi :\bundle{E}\to M\times \bundle{E}_{x,W}$ such that
\[
\bigl\| \nabla -\Phi ^*\nabla ^x\bigl\| _{TM,g}\leq C\cdot \| R^\nabla \| _{TM,g}.
\]
Here, $\nabla ^x=j_x^*\nabla $ is the induced connection on $j_x^*\bundle{E}=M\times \bundle{E}_{x,W}$ where $j_x:M\times W\to M\times W$ is given by $j_x(y,q)=(x,q)$.
\end{thm}
We will only sketch the proof of this theorem because most parts follow by the above observations. At first we adapt proposition \ref{local_uhl}.  We consider only paths which are constant on $W$, i.e.~$\mathrm{Im}(\gamma )\subseteq M\times \{ q\} $ for some $q\in W$. Suppose that $x\in M$ and $r<\varrho _x$ is the injectivity radius at $x$, then the minimal geodesic $\gamma $ in $(M,g)$ from $x$ to $y$ yields a paths in $M\times W$ from $(x,q)$ to $(y,q)$ for all $q\in W$. Hence, the parallel transport $P^{\gamma ,q}_{y,x}:\bundle{E}_{(x,q)}\to \bundle{E}_{(y,q)}$ is a well defined unitary operator for all $q\in W$ where  $\bundle{E}_{(x,q)}$ is the fiber of $\bundle{E}$ in $(x,q)\in M\times W$. Moreover, the parallel transport maps $P^{\gamma ,q} _{y,x}:\bundle{E}_{(x,q)}\to \bundle{E}_{(y,q)}$ extend to smooth sections $W\to \mathrm{Hom}(\bundle{E}_{x,W},\bundle{E}_{y,W}), q\mapsto P_{y,x}^{\gamma ,q}$. Thus, $\Psi :\bundle{E}_{|B_r(x)\times Y}\to B_r(x)\times \bundle{E}_{x,W}$ defined by $\Psi (v)=(\pi _1(v),P^{-\gamma ,\pi _2(v)}_{x,\pi _1(v)}v)$ is a smooth unitary bundle isomorphism (covering the identity on $B_r(x)\times W$) with
\[
\| \nabla -\Psi ^*\nabla ^x\| _{TM,g}\leq C\cdot \| R^\nabla \| _{TM,g}.
\]
Here $\pi _1:\bundle{E}\to M$ and $\pi _2:\bundle{E}\to W$ are given by $\pi (v)=(\pi _1(v),\pi _2(v))\in M\times W$. Note if $V:B_r(x)\times W\to \bundle{E}$ is a section with $\Psi ^*\nabla ^x_wV=0$ for all $w\in TM$, then $V(y,q)=P^{\gamma ,q}_{y,x}V(x,q)$, i.e.~the estimate follows analogously to the proof of proposition \ref{local_uhl}. Let $\{ U_i=B_r(x_i)\} $ be a finite open cover of $M$ for $r<\varrho (M)$. We fix minimal geodesics $\kappa _i$ from $x_i\in M$ to $x_1\in M$ and extend these $\kappa _i$ to $M\times W$ while keeping $\kappa _i$ constant on $W$.  Define $\Psi _i:\bundle{E}_{|U_i\times W}\to U_i\times \bundle{E}_{x_1,W}$ by $\Psi _i(v)=(\pi _1(v),P_{x_1,x_i}^{\kappa _i,\pi _2(v)}P_{x_i,\pi _1(v)}^{-\gamma _i,\pi _2(v)}v)$ for the unique minimal geodesic $\gamma _i$ from $x_i$ to $\pi _1(v)\in B_r(x_i)$. The coordinate transition functions $g_{ij}:(U_i\cap U_j)\times W\to \mathrm{U}(\bundle{E}_{x_1,W})$ are smooth sections given by
\[
g_{ij}(y,q)=\Psi _i\circ \Psi _j^{-1}(y,q)=P_{x_1,x_i}^{\kappa _i,q}P_{x_i,y}^{-\gamma _i,q}P_{y,x_j}^{\gamma _j,q}P_{x_j,x_1}^{-\kappa _j,q}=P_{x_1,x_1}^{\sigma _{ij}(y),q}\in \mathrm{U}(\bundle{E}_{x_1,q})
\]
where $\sigma _{ij}(y)=\kappa _i\circ (-\gamma _i)\circ \gamma _j\circ (-\kappa _j)$ is the loop based at $x_1$, through $y$ and constant at $q$ on $W$. If $\| R^\nabla \| _{TM,g}<\epsilon $, the above estimates and \cite[prop.~3.2]{Uhl1} yield a subcover $\{ V_i\} $ of $M$ with $V_i\subseteq U_i$ and smooth sections $\rho _i:V_i\times W\to \mathrm{U}(\bundle{E}_{x_1,W})$ such that $\mathrm{Id}=\rho _ig_{ij}\rho _j^{-1}$ holds on $(V_i\cap V_j)\times W$ for all $i,j$. We define $\Phi _i:=\rho _i\Psi _i:\bundle{E}_{|V_i\times W}\to V_i\times \bundle{E}_{x_1,W}$, then $\Phi _i\circ \Phi _j^{-1}=\mathrm{Id}$ holds on $(V_i\cap V_j)\times W$. Thus, $\Phi =\{ \Phi _i\} :\bundle{E}\to M\times \bundle{E}_{x_1,W}$ is a unitary bundle isomorphism. Let $\nabla ^{x_1}$ be the induced connection on $M\times \bundle{E}_{x_1,W}$, then the above arguments show
\[
\| \nabla -\Psi _i^*\nabla ^{x_1}\| _{TM,g}\leq C\cdot \| R^\nabla \| _{TM,g}
\]
for all $i$ on $U_i\times W$. Hence,
\[
\Psi _i^*(\nabla ^{x_1}+A_i)=\nabla =\Psi _j^*(\nabla ^{x_1}+A_j)
\]
yields $g_{ji}\nabla ^{x_1}g_{ij}+g_{ji}A_ig_{ij}=A_j$ as well as $\| \nabla ^{x_1}g_{ij}\| _{TM,g}\leq 2C\| R^\nabla \| _{TM,g}$. Following the above estimates proves the claim:
\[
\| \nabla -\Phi ^*\nabla ^{x_1}\| _{TM,g}\leq (C+C'')\| R^\nabla \| _{TM,g}.
\]
\section{Definition of relative K--area homology and properties}
The notion of K--area was introduced by Gromov in \cite{Gr01} and generalized by us in \cite{List10}. Let $(M^n,g)$ be a compact Riemannian manifold (possibly with boundary) and $X\subseteq M$ be a compact subset which has a \emph{geodesic normal neighborhood}, i.e.~we assume the existence of $\delta >0$, such that for all
\[
q\in B_\delta X=\{ p\in M|\ \mathrm{dist}(X,p)<\delta \} 
\]
there is a unique point $q_0\in X$ with $\mathrm{dist}(q,q_0)=\mathrm{dist}(q,X)$. Then the pair $(M,X)$ is a neighborhood deformation retract, hence $(M,X)$ has the homotopy extension property. We assume additionally that $(M,X)$ is a smooth neighborhood deformation retract which means the existence of a smooth map $F^M:M\to M$ with $F_{|X}^M=\mathrm{id}_X$, $F^M(\overline{B}_\delta X)=X$ for some $\delta >0$ and $F^M\simeq \mathrm{id}:(M,X)\to (M,X)$. We consider only pairs $(M,X)$ with these two properties. Note that the existence of a geodesic normal neighborhood and the existence of the smooth map $F^M$ is independent of the Riemannian metric. If $X$ is additionally a submanifold, we call $(M,X)$ a \emph{pair of compact manifolds}. Most compact submanifolds $N\subseteq M$ provide examples of pairs $(M,N) $ by the tubular neighborhood theorem. A continuous map $f:(M,X)\to (N,Y)$ is homotopic to a smooth map $(M,X)\to (N,Y)$ by the following argument: The smooth approximation theorem provides a smooth map $h:M\to N$ homotopic to $f$ and with $h( X)\subseteq \overline{B}_\delta Y$ for $\delta >0$, hence $F^N\circ h:M\to N$ is a smooth map with $F^N\circ h\simeq f:(M,X)\to (N,Y)$, here $F^N:N\to N$ means the smooth neighborhood deformation retract map of the pair $(N,Y)$.

Suppose that $\theta \in H_{2*}(M,X;G )$ is a singular homology class for a coefficient group $G$, then $\mathscr{V}(M,X;\theta )$ denotes the set of all Hermitian vector bundles $\bundle{E}\to M$  with Hermitian connection $\nabla $ such that
\begin{enumerate}
\item[(i)] $(\bundle{E},\nabla )$ is trivial in an open neighborhood of $X$, i.e.~for each path component $M_i$ of $M$ there is an open neighborhood $V_i\subseteq M_i$ of $X\cap M_i$ and a local unitary trivialization $\Psi :\bundle{E} _{|V_i}\to V_i\times \C ^{m_i}$ with $\nabla =\Psi ^*d$ for the canonical trivial connection $d$ on $V_i\times \C ^{m_i}$.
\item[(ii)] the classifying map $\rho ^\bundle{E}:M\to \coprod _m\mathrm{BU}_m $ of the vector bundle $\bundle{E}$ is homological nontrivial with respect to $\theta $, i.e.~$0\neq \rho ^\bundle{E}_*(\theta )$ where $\rho ^\bundle{E}_*:H _{2*}(M,X;G)\to H_{2*}(\coprod _m\mathrm{BU}_m ,\rho ^\bundle{E}(X);G)$, here $\rho ^\bundle{E}(X)\cap \mathrm{BU}_m$ consists of at most one point for all $m$ [this is possible by assumption (i) and the homotopy extension property].
\end{enumerate}
If $g$ is a Riemannian metric on $M$, the \emph{K--area of the homology class} $\theta $ is defined by
\[
\ke (M_g,X;\theta ):=\left( \inf _{(\bundle{E},\nabla )\in \mathscr{V}(M,X;\theta )}\| R^{(\bundle{E},\nabla )}\| _g\right) ^{-1} \in [0,\infty ]
\]
where $\| R^{(\bundle{E},\nabla )}\| _g$ is the usual $L^\infty $ comass operator norm on $\Lambda ^2M\otimes \mathrm{End}(\bundle{E})$. We define the K--area of odd homology classes $\theta \in H_{2*+1}(M,X;G)$ by
\[
\ke (M_g,X;\theta ):=\sup _{dt^2}\ke (M_g\times S^1_{dt^2},X\times S^1;\theta \times [S^1]),
\]
note that $\theta \times [S^1]\in H_{2*}(M\times S^1,X\times S^1;G)$ is well defined if we fix a generator $[S^1]\in H_1(S^1)$. The definitions of the K--area from above and in \cite{List10} coincide for $X=\emptyset $, in fact $\ke (M_g;\theta )=\ke (M_g,\emptyset ;\theta )$ for all $\theta \in H_k(M;G)$. If $M$ is closed and $[M]\in H_n(M)$ a fundamental class, $\ke (M_g;[M])$ is Gromov's K--area introduced in \cite{Gr01}. Since $\ke (M_g,X;\theta )<+\infty $ does not depend on the choice of the Riemannian metric $g$, the K--area homology
\[
\H _k(M,X;G)=\{ \theta \in  H_k(M,X;G)|\ \ke (M_g,N;\theta )<+\infty \}
\]
is independent of the choice of the Riemannian metric. In \cite{List10} we considered the case $X=\emptyset $ and proved plenty of interesting results concerning the K--area homology functor. Omitting the coefficient group means as usual $G=\Z $. The most interesting questions arise for $\Z $ respectively $\Q $ coefficients which provide the same amount of new information by fact (\ref{ite5}) below. However, other coefficient rings (like $\Z _2$) have useful applications as presented by \cite[prop.~6.5]{List10}. In the following we generalize and summarize some of the results from \cite{List10,habil}:
\begin{enumerate}
\item $\H _k(M,X;G)$ is a subgroup, submodule, linear subspace of $H_k(M,X;G)$  if $G$ is an abelian group, a ring respectively a field (cf.~\cite[sec.~2]{List10}).
\item $\H _0(M,X;G)=0$ for all pairs $(M,X)$, in fact the dimension axiom reads $\H _k(\{ pt\} ;G)=0$ for all $k$.
\item \label{ite3} If $f:(M,X)\to (N,Y)$ is continuous, the induced homomorphism of singular homology restricts to homomorphisms
\[
f_*:\H _k(M,X;G)\to \H _k(N,Y;G)
\]
which proves the functoriallity of $\H _*$. Moreover, singular theory implies $f_*=h_*$ if $f\simeq h:(M,X)\to (N,Y)$. 
\item If $f:(M,X)\to (N,Y)$ is a homotopy equivalence of pairs, then $f_*:\H _k(M,X;G)\to \H _k(N,Y;G)$ is an isomorphism for all $k$.
\item $\H _k(M\coprod N;G)=\H _k(M;G)\oplus \H _k(N;G)$.
\item \label{ite5} $\H _k(M,X)\otimes \Q =\H _k(M,X;\Q )$ and $\mathrm{Tor}(H_k(M,X))\subseteq \H _k(M,X)$ hold for all $k$ (cf.~\cite[prop.~2.6]{List10}). A universal coefficient theorem for arbitrary coefficients does not seem to exist for K--area homology. 
\item \label{ite6} If $f:(M,X)\to (N,Y)$ is a \emph{relative diffeomorphism} which means that $f$ is a continuous map of pairs and $f:M\setminus X\to N\setminus Y$ is a diffeomorphism, then
\[
f_*:\H _k(M,X;G)\to \H _k(N,Y;G)
\]
is an isomorphism for all $k$. 
\item If $U\subseteq M$ is an open set such that $M\setminus U$ is a compact manifold, $\overline{U}\subseteq \mathrm{int}(X)$ and $(M\setminus U,X\setminus U)$ is a pair, then the inclusion map is a relative diffeomorphism and provides an isomorphism
\[
\H _k(M\setminus U,X\setminus U;G)\to \H _k(M,X;G).
\] 
\item \label{ite8} $j_*:\H _k(M;G)\to \H _k(M,\{ p\} ;G)$ is an isomorphism for all $k$ and $p\in M$ by theorem \ref{sub_thm} below.
\item \label{ite10}If $X\subseteq M^n$ satisfies $(B_\delta X)\setminus X\approx (0,\delta )\times S^{n-1}$ for small $\delta >0$, then $M/X$ is a compact smooth manifold and $\H _k(M,X;G)\cong \H _k(M/X;G)$ holds for all $k$. Here we use facts (\ref{ite6}) and (\ref{ite8}).
\item Fact (\ref{ite10}) shows $\H _k(D^k,S^{k-1})=\H _k(S^k)=\Z $ if $k\geq 2$, hence the Hurewicz homomorphism satisfies $\pi _k(M,X)\to \H _k(M,X)$ for all $k\geq 2$.
\item \label{ite12} If $N$ is simply connected and closed, then $\H _k(N;G)=H_k(N;G)$ holds for all $k>0$ and coefficient groups $G$. If $N$ is closed with finite fundamental group, then $\H _k(N)=H_k(N)$ for all $k>0$. Remark \ref{rem32} below shows that this fails in general for arbitrary coefficients.
\item $\H _{2k-1}(M)=\{ \theta \in H_{2k-1}(M)|\ \theta \times [S^1]\in \H _{2k}(M\times S^1)\} $ by \cite[cor.~3.5]{List10}. This fails in general for pairs of manifolds, for instance, $\H _1(S^1,\{ x\} )=0$ but $\H _2(T^2, \{ x\} \times S^1)=\Z $. This results from the definition of the K--area for odd homology classes.
\item If $(N,g)$ is a closed Riemannian spin manifold of positive scalar curvature, then $\theta \times [N]\in \H _*(M\times N)$ for all $\theta \in H_*(M)$ (cf.~\cite[prop.~6.1]{List10}).
\item If $N^n$ is a closed manifold with $\H _2(N;\Z _2)=0$ and $\H _n(N)\neq H_n(N)$, then $N$ does not admit a metric of positive scalar curvature (cf.~\cite[prop.~6.5]{List10}).
\item \label{ite15} Let $(M^n,g)$ be closed,  orientable with residually finite fundamental group and sectional curvature $K\leq 0$, then $\H _n(M)=0$ (cf.~\cite[Example (v')]{Gr01}). Moreover, $\H _*(T^n)=0$.
\end{enumerate}
Note that some of the observations (for instance (13),(14)) do not generalize to arbitrary coefficients.
\begin{proof}
(\ref{ite3}): Since $f$ is homotopic to a smooth map $h:(M,X)\to (N,Y)$, the pull back of vector bundles proves that $f_*=h_*$ is well defined (cf.~prop.~2.3 in \cite{List10}).

(\ref{ite6}): $f_*$ is an isomorphism in singular theory. Let $g^M$, $g^N$ be Riemannian metrics on $M$ respectively $N$. We first notice that for all $\epsilon >0$ there is some $\delta >0$ such that $f^{-1}(\overline{B}_{\delta }Y)\subseteq \overline{B}_\epsilon X$. Here we use that $M$ and $N$ are compact, $f:M\to N$ is continuous and $f:M\setminus X\to N\setminus Y$ is a homeomorphism. Hence, we can choose $\delta >0$ such that $(N,Y)\hookrightarrow (N,\overline{B}_\delta Y)$ and $(M,X)\hookrightarrow (M,X')$ induce isomorphisms on their K--area homology groups where $X'=f^{-1}(\overline{B}_\delta Y)$. By construction there is a constant $C>0$ such that $\frac{1}{C}g^M\leq f^*g^N\leq C\cdot g^M$ holds on $\overline{M\setminus X'}$ and since $f^*:\mathscr{V}(N,\overline{B}_\delta Y;f_*\theta )\to \mathscr{V}(M,X';\theta )$ is a bijection, we conclude
\[
\frac{1}{C}\ke (M,X';\theta )\leq \ke (N,\overline{B}_{\delta }Y;f_*\theta )\leq C\cdot \ke (M,X';\theta )
\]
for all $\theta \in H_{2*}(M,X';G)$. The same arguments apply to products with $S^1$ which provides the same inequalities for $\theta \in H_{2*+1}(M,X';G)$. Thus, $f_*:\H _k(M,X';G)\to \H _k(N,\overline{B}_\delta Y;G)$ is an isomorphism.

(\ref{ite12}): Choose $\epsilon =\epsilon (N,g)>0$ from theorem \ref{thm_uhl}. If $(\bundle{E},\nabla )\to N$ (respectively $(\bundle{E},\nabla )\to N\times S^1$) is a bundle with curvature $\| R^\nabla \| <\epsilon $, then theorem \ref{thm_uhl} (respectively \ref{thm_uhl2}) shows that $\bundle{E}$ is trivial. Hence, the classifying map $\rho ^\bundle{E}$ is homotopic to the constant map which yields $\rho ^\bundle{E}_*(\theta )=0$ (respectively $\rho ^\bundle{E}_*(\theta \times [S^1])=0$) for all $\theta \in H_k(N;G)$ if $k>0$. This proves $\ke (N_g;\theta )\leq \frac{1}{\epsilon }$. If $N$ has finite fundamental group, consider the universal covering $p:\tilde N\to N$ and the epimorphism $p_*:H_k(\tilde N;\Q )\to H_k(N;\Q )$.
\end{proof}

\begin{thm}[\cite{habil}]
$\H _*(.;G)$ is a functor on the category pairs of compact smooth manifolds and continuous maps into the category of graded abelian groups which satisfies the Eilenberg--Steenrod axioms up to the existence of long exact homology sequences. Moreover, the suspension functor $\Sigma ^k$ stabilizes the functor of K--area homology into the functor of singular homology if $k\geq 2$.
\end{thm}
Our main theorem \ref{main_thm} provides the existence of very short exact homology sequences for pairs $(M,N)$ if all components of $N$ have a finite fundamental group. We refer to \cite[theorem 6.4]{List10} for the claim that the suspension functor $\Sigma ^k$ stabilizes K--area homology into singular homology if $k\geq 2$. The relative version to \cite[theorem 6.4]{List10} follows considering the quotient space $M/X$ instead of the pair $(M,X)$. An alternative approach is presented in \cite[sec.~1.6]{habil}.  The main ingredient in these proofs is the observation in the following proposition, in fact $\eta \times [S^k]\in \H _{j+k}(M\times S^k,X\times S^k;G)$ for all $\eta \in H_j(M,X;G)$ if $k\geq 2$.

\begin{prop}
\begin{enumerate}
\item[(i)] Let $\F $ be a field or $\F =\Z $. If $N$ is simply connected, closed and $\eta \in H_k(N;\F )$ for $k>0$, then $\theta \times \eta \in \H _{j+k}(M\times N,X\times N;\F )$ holds for all $\theta \in H _j(M,X;\F )$.
\item[(ii)] Suppose that $k\geq 2$. Fix a point $e_k\in S^k$, a fundamental class $[S^k]\in H_k(S^k)$ and let $i:M\to M\times S^k$ be the inclusion map $i(p)=(p,e_k)$, then  
\[
\begin{split}
\H _{j+k}(M,X;G)\oplus H_j(M,X;G)&\to \H _{j+k}(M\times S^k,X\times S^k;G),\\
(\eta ,\theta )&\mapsto i_*\eta +\theta \times [S^k]
\end{split}
\]
is a natural isomorphism for all $j$.
\item[(iii)] If $N$ is connected and closed with finite fundamental group and $M$ is compact, then the homology cross product provides an isomorphism
\[
\H _k(M\times N;\Q )=\H _k(M;\Q )\oplus \bigoplus _{i=2}^kH _{k-i}(M;\Q )\otimes H_i(N;\Q ).
\]
\end{enumerate}
\end{prop}
\begin{proof}
(i) If $\theta \times \eta $ has infinite K--area, there is a sequence of bundles $\bundle{E}_i\to M\times N$ (respectively  $\bundle{E}_i\to M\times N\times S^1$) with $\rho ^{\bundle{E}_i}_*(\theta \times \eta )\neq 0$ (respectively $\rho ^{\bundle{E}_i}_*(\theta \times \eta \times [S^1])\neq 0$) and with curvature $\| R^{\bundle{E}_i}\| <\frac{1}{i}$. Theorem \ref{thm_uhl2} implies the existence of $i_0$ such that for all $i\geq i_0$ there is a bundle $\bundle{F}_i$ on $M$ (respectively on $M\times S^1$) with $\bundle{E}_i\cong N\times \bundle{F}_i$. Thus, we can choose the classifying map $\rho ^{\bundle{E}_i}$ to be constant on $N$, i.e.~$\rho ^{\bundle{E}_i}_*(\theta \times \eta )=0$ (respectively $\rho ^{\bundle{E}_i}_*(\theta \times \eta \times [S^1])= 0$) for all $i\geq i_0$ which yields a contraction. Hence $\theta \times \eta $ has finite K--area. 
 
(ii) The map $(\eta ,\theta )\mapsto i_*\eta +\theta \times [S^k]$ is an isomorphism for singular homology. Using the projection map $M\times S^k\to M$ and the functoriallity of $\H _*$, then $i_*\eta $ has finite K--area if and only if $\eta \in H _{j+k}(M,X;G)$ has finite K--area. Moreover, $\theta \in [S^k]$ has finite K--area for all $\theta \in H_j(M,X;G)$ by the same arguments as in the proof of (i): If $\bundle{E}\to M\times S^k$ has sufficiently small curvature, it is induced from a bundle on $M$ by theorem \ref{thm_uhl2} which means that $\rho ^{\bundle{E}}_*(\theta \times [S^k])=0$, a contradiction to the assumption that $\theta \times [S^k]$ has infinite K--area.

(iii) Use the K\"unneth theorem for rational singular homology. If $\eta \in H_0(N;\Q )$, then $\theta \times \eta \in \H _k(M\times N;\Q )$ if and only if $\theta \in \H _k(M;\Q )$ by \cite[prop.~2.3 and prop.~5.1]{List10}. Given $\eta \in H_j(N;\Q )$ for some $j>1$, then $\eta =f_*(\tilde \eta )$ for some $\tilde \eta \in H_j(\tilde N;\Q )$ where $f:\tilde N\to N$ is the universal covering. Since $\tilde \eta \times \theta \in \H _{k}(\tilde N\times M;\Q )$ for all $\theta \in H_{k-j}(M;\Q )$ by the results in (i), we conclude $\eta \times \theta =(f\times \mathrm{id})_*(\tilde \eta \times \theta )\in \H _k(N\times M;\Q )$.
\end{proof}
\begin{rem}
\label{rem32}
$\H _{2*}(\R P^n;\Z _2)=0$ for all $n$.
\end{rem}
\begin{proof}
Take the complex line bundle $\bundle{L}\to \R P^n$ with Chern class $0\neq c_1(\bundle{L})\in H^2(\R P^n;\Z )=\Z _2$, then $0\neq c_1(\bundle{L})\in H^2(\R P^n;\Z _2)$ and $\bundle{L}$ admits a flat connection by $H^2(\R P^n;\R )=0$. Hence, the bundle $\bundle{E}=\bundle{L}^{\oplus k}$ satisfies $c_k(\bundle{E})=c_1(\bundle{L})^k=\alpha ^{2k}$ which implies $\left< c_k(\bundle{E}),\beta \right> \neq 0$ for $0\neq \beta \in H_{2k}(\R P^n;\Z _2)$. Since $\bundle{E}$ is flat, this yields $\H _{2k}(\R P^n;\Z _2)=0$ for all $k$. 
\end{proof}
A complex vector bundle $\bundle{E}\to M$ is \emph{stably rational trivial} if $\bundle{E}^{\oplus m}\oplus  \C ^l$ is a trivial bundle for some integers $m,l$. In fact, $\bundle{E}$ is stably rational trivial if and only if $[\bundle{E}]=[\C ^j]$ holds in $ K(M)\otimes \Q $ where $j=\mathrm{rk}(\bundle{E})$. We observed that $\H _{k}(M)=H_{k}(M)$ holds for all $k>0$ if $M$ is a closed manifold with finite fundamental group. However, there are plenty of closed manifolds $M$ with infinite fundamental group and $\H _k(M)=H_k(M)$ for all $k>0$ (for instance $M=\R P^n\# \R P^n$, $n\geq 3$). The following proposition shows the interest in manifolds whose K--area homology equals the reduced singular homology. 
\begin{prop}
Let $(M,g)$ be a compact Riemannian manifold, then the following is equivalent:
\begin{enumerate}
\item[(i)] There is an $\epsilon >0$ with the following property: Any Hermitian vector bundle $(\bundle{E},\nabla )\to M$ with curvature $\| R^{(\bundle{E},\nabla )}\| _g<\epsilon $ is stably rational trivial. 
\item[(ii)] $\H _{2k}(M)=H_{2k}(M)$ for all $k>0$.
\end{enumerate}
\end{prop}
\begin{proof}
(i)$\Rightarrow $(ii): $\bundle{E}$ is stably rational trivial if and only if $\mathrm{ch}(\bundle{E})=\mathrm{rk}(\bundle{E})$, hence $\ke (M_g;\theta )\leq \frac{k^2}{\epsilon }$ for all $\theta \in H_{2k}(M;\Q )$ by \cite[prop.~3.1]{List10}.\\
(ii)$\Rightarrow $(i): Let $\theta _1,\ldots ,\theta _s$ be a basis of $\H _{2*}(M;\Q )$ with $\ke (M_g;\theta _i)\leq \ke (M_g;\theta _s)$ for all $i$, then \cite[prop.~2.2]{List10} and the assumption in (ii) imply $\ke (M_g;\theta )\leq \ke (M_g;\theta _s)<\infty $ for all $\theta \in H_{2k}(M;\Q )$ if $k>0$. Set $\epsilon =1/\ke (M_g;\theta _s)$, then $\| R^{(\bundle{E},\nabla )}\| _g<\epsilon $ yields $\left< c_k(\bundle{E}),\theta \right> =0$ for all $\theta \in H_{2k}(M;\Q )$ if $k>0$.
\end{proof}

\section{Proof of the main theorem}
\label{sect4}
Let $(M,g)$ be a compact Riemannian manifold and $N\subseteq M$ be a closed submanifold with finite fundamental group on each path component. We choose $\delta =\delta (M,g)>0$ such that $B_sN=\{ p\in M\ |\ \mathrm{dist}(p,N)<s\} $ is a tubular (or collar) neighborhood of $N$ for $s=4\delta $. Let $f:M\to M$ be a smooth map which is homotopic to the identity $\mathrm{id}:M\to M$ and satisfies $f(B_{3\delta }N)=N$ as well as $f_{|N}=\mathrm{id}_N$. Moreover, let $h:B_{3\delta }N\to [0,1]$ be a smooth map with $h(p)=1$ if $\mathrm{dist}(p,N)\in [2\delta ,3\delta )$ and $h(p)=0$ for $p\in B_{\delta }N$.  We use the constants $C=C(\tilde N,\tilde g)<\infty $ and $\epsilon =\epsilon (\tilde N,\tilde g)>0$ from theorem \ref{thm_uhl} for the universal covering $(\tilde N,\tilde g)$ of the Riemannian manifold $(N,g_{|N})$, but without loss of generality choose $\epsilon \leq \frac{1}{2C^2}$. Moreover, we define the constant
\[
\mathfrak{c}:=\left( C\cdot \| dh\| _g+1\right) \cdot \| df\| _g^2\in (1,\infty ).
\]
If the normal bundle to $N\subseteq M$ is trivial, then for a given Riemannian metric $g'$ on $N$ and all $s >0$ there is a metric $g$ on $M$ such that $g_{|N}=g'$ and $\| df\| _g=1$, $\| dh\| _g\leq s/C$, in fact $\mathfrak{c} \in (1,1+s] $. 
\begin{lem} \label{lem41}Suppose that $\pi _1(N,x)=0$ for all $x\in N$.\\
a) Let $(\bundle{E},\nabla )\to M$ be a Hermitian vector bundle with curvature $\| R^\nabla \| _g<\frac{\epsilon }{\| df\| _g^2}$, then there is a Hermitian connection $\widetilde{\nabla }$ on $\bundle{E}$ such that
\begin{enumerate}
\item $\| R^{\widetilde{\nabla }}\| _g\leq \mathfrak{c} \cdot \| R^\nabla \| _g$.
\item $(\bundle{E},\widetilde{\nabla })$ is trivial in a neighborhood of $N$, in fact there is a trivialization $\Psi :\bundle{E}_{|B_{\delta }N}\to B_{\delta }N \times \C ^m$ with $\widetilde{\nabla }=\Psi ^*d$ on $B_{\delta }N$.
\end{enumerate}
b) If $(\bundle{E},\nabla )\to M\times S^1$ has curvature $\| R^\nabla \| _{TM,g}<\frac{\epsilon }{\| df\| _g^2}$, then there is a Hermitian connection $\widetilde{\nabla }$ on $\bundle{E}$ such that
\begin{enumerate}
\item $\| R^{\widetilde{\nabla }}\| _{TM,g}\leq \mathfrak{c} \cdot \| R^\nabla \| _{TM,g}$.
\item $(\bundle{E},\widetilde{\nabla })$ is trivial in a neighborhood of $N\times S^1$, i.e.~there is a trivialization $\Psi :\bundle{E}_{B_\delta N\times S^1}\to B_\delta N\times S^1\times \C ^m$ with $\widetilde{\nabla }=\Psi ^*d$ on $B_\delta N\times S^1$.
\end{enumerate}
\end{lem}
\begin{proof}
a) We consider $(f^*\bundle{E},f^*\nabla )\to M$ and obtain $\| R^{f^*\nabla }\| _g\leq \| df\| _g^2\cdot \| R^\nabla \| _g<\epsilon $. Since $f\simeq \mathrm{id}$ there is a unitary bundle isomorphism $\Upsilon :\bundle{E}\to f^*\bundle{E}$. On $B_{3\delta }N$ the Hermitian bundle $(f^*\bundle{E},f^*\nabla )$ is the pull back of $(\bundle{E}_{|N},\nabla )\to N$, hence we can apply theorem \ref{thm_uhl}. In particular, there is a trivialization $\Phi :f^*\bundle{\bundle{E}}_{|B_{3\delta }N}\to B_{3\delta }N\times \C ^m$ such that $A:=f^*\nabla -\Phi ^*d$ satisfies $\| A\| _g\leq C\cdot \| R^{f^*\nabla }\| _g$ on $B_{3\delta }N$. Moreover, $R^{f^*\nabla }=dA+A\wedge A$ yields
\[
\| dA\| _g\leq \| R^{f^*\nabla }\| _g+2\| A\| _g^2\leq (1 +2C^2\| R^{f^*\nabla }\| _g)\| R^{f^*\nabla }\| _g\leq 2\| R^{f^*\nabla }\| _g.
\]
Define the connection $\overline{\nabla }$ on the bundle  $f^*\bundle{E}$ by:
\[
\overline{\nabla }=\Phi ^*d+h\cdot A \qquad \text{on}\quad B_{3\delta }N
\]
and by $\overline{\nabla }=f^*\nabla $ on $M\setminus B_{3\delta }N$. Then $\overline{\nabla }$ is certainly a smooth connection with $\| R^{\overline{\nabla }}\| _g\leq \| df\| _g^2\cdot \| R^\nabla \| _g$ on $M\setminus B_{3\delta }N$ and
\[
\begin{split}
\| R^{\overline{\nabla }}\| _g&=\| d(hA)+h^2A\wedge A\| _g\leq \| dh \| _g\cdot \| A\| _g+\|h^2R^{f^*\nabla }+h(1-h)dA\| _g\\
&\leq \left( C\| dh\| _g+\max [h^2+2h(1-h)]\right) \| R^{f^*\nabla }\| _g=\mathfrak{c} \cdot \| R^\nabla \| _{g} 
\end{split}
\]
on $B_{3\delta }N$. Moreover, by construction $\overline{\nabla }=\Phi ^*d$ holds on $B_\delta N$ and thus, $\widetilde{\nabla }:=\Upsilon ^*\overline{\nabla }$ is a Hermitian connection on $\bundle{E}$ which satisfies the conditions (1) and (2). Notice that $\Psi =\Phi \circ \Upsilon $ provides the corresponding local trivialization of $\bundle{E}_{|B_{\delta }N}$.

b) Consider the map $\tilde f=f\times \mathrm{id}:M\times S^1\to M\times S^1$, then $\tilde f$ is homotopic to $\mathrm{id}_{M\times S^1}$, i.e.~there is a bundle isomorphism $\Upsilon :\bundle{E}\to \tilde f^*\bundle{E}$. On $B_{3\delta }N\times S^1$ the bundle $(\tilde f^*\bundle{E},\tilde f^*\nabla )$ is the pull back of $\bundle{E}_{|N \times S^1}$, hence theorem \ref{thm_uhl2} provides a trivialization $\Phi :\tilde f ^*\bundle{E}_{B_{3\delta }N\times S^1}\to B_{3\delta }N\times S^1\times \C ^m$ such that $A:=\tilde f^*\nabla -\Phi ^*d$ satisfies $\| A\| _{TM,g}\leq C\cdot \| R^{\tilde f^*\nabla }\| _{TM,g}$ on $B_{3\delta }N\times S^1$. Here we use that any complex bundle on $S^1$ is trivial, in fact $\nabla ^y-d\in \Omega ^1S^1\otimes \frak{u}(\tilde f^*\bundle{E}_{y,S^1})$ for $y\in N$. Define the connection $\overline{\nabla }=\Phi ^*d+h\cdot A$ on $B_{3\delta }N\times S^1 $ and $\overline{\nabla }=\tilde f^*\nabla $ outside $B_{3\delta }N\times S^1$, then
\[
\| R^{\overline{\nabla }}\| _{TM,g}\leq \mathfrak{c} \| R^\nabla \| _{TM,g}
\]
holds by the estimates in a). Hence $\widetilde{\nabla }:=\Upsilon ^*\overline{\nabla }$ satisfies the condition (1), (2) and $\Psi =\Phi \circ \Upsilon $. Notice that we have no control of $\| A\| _{TS^1,dt^2}$ in terms of the curvature. 
\end{proof}
\begin{lem} \label{lem42}
Suppose that $\pi _1(N,x)$ is finite for all $x\in N$. Let $(\bundle{E},\nabla )\to M$ (respectively $(\bundle{E},\nabla )\to M\times S^1$) be a Hermitian bundle with $\| R^\nabla \| _{TM,g}<\frac{\epsilon }{\| df\| _g^2}$, then there is a Hermitian bundle $(\widetilde{\bundle{E}},\widetilde{\nabla })$ on $M$ (resp.~on $M\times S^1$) such that
\begin{enumerate}
\item $\widetilde{\bundle{E}}=\bigoplus ^m\bundle{E}$ for some integer $m>0$.
\item $\| R^{\widetilde{\nabla }}\| _{TM,g}\leq \mathfrak{c} \cdot \| R^\nabla \| _{TM,g}$.
\item $(\widetilde{\bundle{E}},\widetilde{\nabla })$ is trivial in a neighborhood of $N$ (respectively $N\times S^1$).
\end{enumerate}
\end{lem}
\begin{proof}
We mention only the modifications to the last proof. Let $N$ be connected and consider the restriction $(\bundle{F},\nabla )=(\bundle{E},\nabla )_{|N}$ as well as the universal covering $p:\bar N\to N$, then $p^*\bundle{F}$ is trivial by theorem \ref{thm_uhl}, in fact there is a trivialization $\bar \Psi :p^*\bundle{F}\to \bar N\times \C ^k$ with $\| \bar \nabla -\bar \Psi ^*d\| \leq C\cdot \| R^{\nabla }\| _g$ where $\bar \nabla $ is induced by $\nabla $. The push forward bundle $p_{!}(p^*\bundle{F},\bar \nabla )\to N$ is (up to an isomorphism) given by $\bigoplus ^m(\bundle{F},\nabla )$ where $m=| \pi _1(N)|$. Moreover, $\bar \Psi $ determines a trivialization of $p_{!}p^*\bundle{F}$ by
\[
(p_!p^*\bundle{F})_y=\bigoplus _{x\in p^{-1}(y)}(p^*\bundle{F})_x\ni (w_1,\ldots ,w_m)\mapsto (\mathrm{pr}_2\bar \Psi (w_1),\ldots ,\mathrm{pr}_2\bar \Psi (w_m))\in \C ^{m\cdot k}.
\]
This yields a trivialization $\Psi :\bigoplus ^m\bundle{F}\to N\times \C ^{m\cdot k} $ with $\| \nabla -\Psi ^*d\| _g=\| \bar \nabla -\bar \Psi ^*d\| _g\leq C\cdot \| R^\nabla \| _g$. Thus, in place of $\bundle{E}$ we use the bundle $\widetilde{\bundle{E}}:=\bigoplus ^m\bundle{E}$ on $M$ and a suitable trivialization near $N$ to conclude the claim. Then we proceed as in the proof of the last lemma. If $N$ is not connected, define $m$ to be the lowest common multiple of the integers $|\pi _1(N,x_1)|,\ldots ,| \pi _1(N,x_s)|$ where $x_1,\ldots ,x_n$ represent $\pi _0(N)$.  
\end{proof}
\begin{cor}
(i) Suppose that all components of $N$ have trivial fundamental group, then $\ke (M_g;\theta )>\frac{\| df\| _g^2}{\epsilon }$ implies
\[
\ke (M_g;\theta )\leq \mathfrak{c} \cdot \ke (M_g,N;j_*\theta )
\]
for all $\theta \in H_{k}(M;G)$ and $k>0$ where $j:M\to (M,N)$.\\
(ii) If $|\pi _1(N,x)|<\infty $ for all $x\in N$, then $\ke _\mathrm{ch}(M_g;\theta )>\frac{\| df\| _g^2}{\epsilon }$ implies
\[
\ke _\mathrm{ch}(M_g;\theta )\leq \mathfrak{c} \cdot \ke (M_g,N;j_*\theta )
\]
for all $\theta \in H_{k}(M;\Q )$ and $k>0$. Here $\ke _\mathrm{ch}$ denotes the K--area for the Chern character (cf.~\cite[sec.~3]{List10}). 
\end{cor}
\begin{proof}
(i) We start with the case $k\in 2\Z $. Let $(\bundle{E},\nabla )\to M$ be a Hermitian bundle of constant rank $m$, curvature $\| R^\nabla \| _g<\frac{\epsilon }{\| df\| _g^2}$ and $\rho ^\bundle{E}_*(\theta )\neq 0$ for the classifying map $\rho ^\bundle{E}:M\to \mathrm{BU}_m $. Then the restriction of $\bundle{E}$ to $N$ is trivial by theorem \ref{thm_uhl}, i.e.~$\bundle{E}$ can be regarded as a relative bundle and we can choose the classifying map $\rho ^\bundle{E}$ to be constant $u\in \mathrm{BU}_m$ on $N$. In particular, $\rho ^\bundle{E}_*:H_k(M,N;G)\to H_k(\mathrm{BU}_m ,u;G)$ satisfies $\rho ^\bundle{E}_*j_*(\theta )=j'_*\rho ^\bundle{E}_*(\theta )\neq 0$ for $j':\mathrm{BU}_m \to (\mathrm{BU}_m ,u)$. Lemma \ref{lem41}a) yields a connection $\widetilde{\nabla }$ which is trivial on a neighborhood of $N$ and with $\| R^{\widetilde{\nabla }}\| _g\leq \mathfrak{c} \| R^\nabla \| _g$, i.e.~$(\bundle{E},\widetilde{\nabla })\in \mathscr{V}(M,N;j_*\theta )$. Considering the case $\| R^{(\bundle{E},\nabla )}\| _g\to 1/\ke (M_g;\theta )$ completes the proof. 

Suppose now that $k $ is odd. Let $(\bundle{E},\nabla )\to M\times S^1$ be a Hermitian bundle of constant rank, curvature $\| R^\nabla \| _{g\oplus dt^2}<\frac{\epsilon }{\| df\| ^2_g}$ for some line element $dt^2$ of $S^1$ and $\rho ^\bundle{E}(\theta \times [S^1])\neq 0$. By lemma \ref{lem41}b) we can choose the classifying map $\rho ^\bundle{E}$ to be constant on $N\times S^1$ and $\rho ^\bundle{E}_*(j_*\theta \times [S^1])=j'\rho ^\bundle{E}_*(\theta \times [S^1])\neq 0$. Moreover, the lemma also provides a connection $\widetilde{\nabla }$ on $\bundle{E}$ which is trivial in a neighborhood of $N\times S^1$ and with $\| R^{\widetilde{\nabla }}\| _{TM,g}\leq \mathfrak{c} \| R^\nabla \| _{g\oplus dt^2}$. Hence, $(\bundle{E},\widetilde{\nabla })\in \mathscr{V}(M\times S^1,N\times S^1;\theta \times [S^1])$ and since the value $\alpha :=\| R^{\widetilde{\nabla }}\| _{g\oplus dt^2}$ is finite, 
\[
\| R^{\widetilde{\nabla }}\| _{g\oplus r^2\cdot dt^2}\leq \max \left\{ \mathfrak{c} \| R^\nabla \| _{g\oplus dt^2} ,\frac{\alpha }{r}\right\} \leq \mathfrak{c} \| R^\nabla \| _{g\oplus dt^2}
\] 
holds for  $\frac{\alpha }{\mathfrak{c} \| R^{\nabla }\| } \leq r<\infty   $. This  proves
\[
\ke (M_g\times S^1_{dt^2};\theta \times [S^1])\leq \mathfrak{c} \cdot  \sup _{ds^2}\ke (M_g\times S^1_{ds^2},N\times S^1;\theta \times [S^1])
\]
for all line elements $dt^2$ on $S^1$.

(ii) This follows analogously to case (i). The K--area for the Chern character is needed to compensate the usage of the bundle $\bigoplus ^m\bundle{E}$ instead of $\bundle{E}$. However, according to \cite[prop.~3.1]{List10} the K--area for the Chern character determines the rational K--area homology uniquely. We start with a bundle $(\bundle{E},\nabla )\to M$ of curvature $\| R^\nabla \| _g<\frac{\epsilon }{\| df\| _g^2}$ and $\left< \mathrm{ch}(\bundle{E}),\theta \right> \neq 0$. Consider the associated bundle $(\widetilde{\bundle{E}},\widetilde{\nabla })$ from lemma \ref{lem42}, then $\left< \mathrm{ch}(\widetilde{\bundle{E}}),\theta \right> =m\cdot \left< \mathrm{ch}(\bundle{E}),\theta \right> \neq 0$ proves $\rho ^{\widetilde{\bundle{E}}}(\theta )\neq 0$. Hence, $\| R^{\widetilde{\nabla }}\| _g\leq \mathfrak{c}\| R^\nabla \| _g$ and $(\widetilde{\bundle{E}},\widetilde{\nabla })\in \mathscr{V}(M,N;j_*\theta )$ yield the inequality if $k\in 2\Z $ [consider $\| R^{(\bundle{E},\nabla )}\| _g\to 1/\ke _\mathrm{ch}(M_g;\theta )$]. The remaining case is left to the reader.
\end{proof}

\begin{thm}
\label{sub_thm}
If $N\subseteq M$ is a closed submanifold with $\pi _1(N,x)=0$ for all $x\in N$, then 
\[
\stackrel{\partial _*}\longrightarrow \H _k(N;G)\stackrel{i_*}\longrightarrow \H _k(M;G)\stackrel{j_*}\longrightarrow \H _k(M,N;G)\stackrel{\partial _*}\longrightarrow \H _{k-1}(N;G) \stackrel{i_*}\longrightarrow 
\]
is a well defined chain complex with $\mathrm{Im}(i_*)= \ker (j_*)$ and $\mathrm{Im}(j_*)=\ker (\partial _*)$ for all $k$ and coefficient groups $G$. Moreover, the same is true if all path components of $N$ have finite fundamental group and $G\in \{ \Z ,\Q \} $. 
\end{thm}
\begin{proof}
If all components of $N$ are simply connected, $\H _k(N;G)=H_k(N;G)$ holds for all $k>0$. Thus, the sequence is a well defined chain complex for $k>1$. Moreover, $\mathrm{Im}(i_*)\subseteq \ker (j_*)$ and $\H _k(N;G)=H_k(N;G)$ yield $\mathrm{Im}(i_*)=\ker (j_*)$. It remains to show $\mathrm{Im}(j_*)=\ker (\partial _*)$. Again ''$\subseteq $'' follows from singular theory. Now suppose $\theta \in \H _k(M,N;G)\cap \ker (\partial _*)$, then there is some $\eta \in H_k(M;G)$ with $j_*\eta =\theta $. The last corollary proves $\eta \in \H _k(M;G)$, i.e.~$\mathrm{Im}(j_*)=\ker (\partial _*)$. The case $k=0$ follows from $\H _0(M,X;G)=\{ 0\} $ for all $(M,X)$. The case $k=1$ follows analogously to $k>1$ if $\partial _*:\H _1(M,N;G)\to \H _0(N;G)=0$ is well defined. The universal coefficient theorem implies $H_1(M,N;G)=H_1(M,N)\otimes G$, i.e.~$\partial _*(\mathrm{Tor}(H_1(M,N))\otimes G)=0 $ where $\mathrm{Tor}(H)$ denotes the torsion subgroup of $H$. If $\theta \notin H_1(M,N;G)\setminus \mathrm{Tor}(H_1(M,N))\otimes G$, there is a map $f:(M,N)\to (S^1,x)$ with $f_*(\theta )\neq 0$. Hence, $\H _1(S^1,x;G)=0$ and the functoriallity of $\H $ prove $\H _1(M,N;G)\subseteq \mathrm{Tor}(H_1(M,N))\otimes G$, i.e.~$\partial _*:\H _1(M,N;G)\to \H _0(N;G)$ is well defined.

Suppose now that $\pi _1(N,x)$ is finite for all $x\in N$, then $\H _k(N;\Q )=H_k(N;\Q )$ holds for all $k>0$ by fact (\ref{ite12}) in section 3. Hence, the above arguments, \cite[prop.~3.1]{List10} and the last corollary provide the claim. The case $G=\Z $ follows immediately from fact (\ref{ite5}) in section 3. 
\end{proof} 
\begin{example}
If the submanifold is not simply connected, plenty of counterexamples to theorems \ref{sub_thm} and \ref{main_thm} exist:
\begin{enumerate}
\item[a)]  We know $\H _*(S^1)=0$ and $\H _*(T^2)=0$ whereas $\H _2(T^2,S^1\times \{ x\} )=\Z $, in fact $\partial _*:\H _2(T^2, \{ x\} \times S^1)\to \H _1(S^1)$ respectively the long homology sequence are well defined but $0=\mathrm{Im}(j_*)\neq \ker (\partial _*)=\Z $.
\item[b)] If $S^1\subset T^2$ bounds a disc, then $\H _2(T^2,S^1)=\Z $ and $0\neq \partial _*:H_2(T^2,S^1)\to H_1(S^1)$ show that $\partial _*$ does not restrict to a map $\H _2(T^2,S^1)\to \H _1(S^1)$, i.e.~the sequence in theorem \ref{main_thm} is not defined. 
\item[c)] Theorem \ref{sub_thm} does not hold for $\Z _2$--coefficients if the submanifold has a nontrivial finite fundamental group: The pair $\R P^{n-1}\subset\R P^{n}$ satisfies $\R P^n/\R P^{n-1}\cong S^n$ which yields $\H _n(\R P^n,\R P^{n-1};\Z _2)=\Z _2$ if $n\geq 2$. Hence,
\[
\qquad 0\longrightarrow \H _n(\R P^n;\Z _2)\longrightarrow \H _n(\R P^n,\R P^{n-1};\Z _2)\stackrel{\partial _*=0}\longrightarrow \H _{n-1}(\R P^{n-1};\Z _2)
\]
can not be exact by remark \ref{rem32} if $n\geq 2$ is even.
\end{enumerate}
\end{example}

\section{Relative K--area and the APS index theorem}
Let $(M^n,g)$ be a connected compact Riemannian manifold with boundary $\partial M$ such that $g=dt^2\oplus g_{\partial M}$ holds near $\partial M$, and let $\Di  ^+:\Gamma (\bundle{S}^+)\to \Gamma (\bundle{S}^-)$ be a complex Dirac operator on $M$ which satisfies the APS boundary conditions (cf.~\cite{APS,APS2,Go2}). The Atiyah--Patodi--Singer index theorem says
\[
\mathrm{ind}_\mathrm{APS}(\Di  ^+)+\frac{h+\eta }{2}=\int _M\alpha _0
\]
where $\eta =\eta (0)$ is the $\eta $--invariant of the associated boundary operator $\Di  _{\partial M}$, $h=\dim \ker (\Di  _{\partial M})$ and $\alpha _0\in \Omega ^n(M) $ is a differential form which depends on the Riemannian metric $g$ and the Dirac bundle $\bundle{S}$:
\[
\alpha _0=\widehat{A}(TM,\nabla ^{TM})\cdot \mathrm{ch}(\bundle{S}/\spinor ,\nabla ^\bundle{S}).
\]
Let $M$ be a spin manifold and $\spinor M$ be the complex spinor bundle. Suppose that $(\bundle{E},\nabla ^{\bundle{E}})$ is a Hermitian vector bundle on $M$ which is trivial on $[0,\epsilon )\times \partial M$, in fact $\nabla ^\bundle{E}$ coincides with the canonical trivial connection for a trivialization on $[0,\epsilon )\times \partial M$. Then $\bundle{S}=\spinor M\otimes \bundle{E}$ is a Dirac bundle on $M$ which satisfies the APS boundary conditions and $\widetilde{\mathrm{ch}}(\bundle{E}):=\mathrm{ch}(\bundle{E})-\mathrm{rk}(\bundle{E})\in H^{2*}(M,\partial M;\Q )$ yields
\[
\alpha _0=\mathrm{rk}(\bundle{E})\cdot \widehat{A}_{n/4}(TM,\nabla ^{TM})+\widehat{A}(TM)\smallsmile \widetilde{\mathrm{ch}}(\bundle{E})
\]
where $\widehat{A}(TM)\in H^{4*}(M;\Q )$ is the $\widehat{A}$--class determined by the Pontryagin classes of $TM$ (independent of the Riemannian metric) and $\widehat{A}_{n/4}(TM,\nabla ^{TM})\in \Omega ^n(M)$ is the differential form induced by the $\widehat{A}_{n/4}$--polynomial in the curvature of $\nabla ^{TM}$ with $\widehat{A}_{n/4}=0$ if $n/4$ is not an integer. Note that $\widehat{A}(TM)\smallsmile \widetilde{\mathrm{ch}}(\bundle{E})\in H^{2*}(M,\partial M;\Q )$ is the relative cup product. If $\dirac _{\partial M}$ is the spin Dirac operator associated to the complex spinor bundle $\spinor (\partial M)=(\spinor ^+M)_{|\partial M}$, the boundary operator $\Di _{\partial M}$ is given by $m\cdot \dirac _{\partial M}$ where $m$ denotes the rank of $\bundle{E}$. Hence, the Dirac operator $\Di ^+:\Gamma (\spinor ^+M\otimes \bundle{E})\to \Gamma (\spinor ^-M\otimes \bundle{E})$ has  index
\[
\mathrm{ind}_\mathrm{APS}(\Di ^+)=\left< \widetilde{\mathrm{ch}}(\bundle{E}),\widehat{A}(TM)\cap [M]\right> -m\left( \frac{\eta +h}{2}-\int _M\widehat{A}_{n/4}(TM,\nabla ^{TM})\right) 
\]
where $\eta =\eta (\dirac _{\partial M})$, $h=\dim \ker \dirac _{\partial M}$ and $\cap [M]:H^{n-k}(M;\Q )\to H_k(M,\partial M;\Q )$ is the Poincar\'e --Lefschetz duality map, i.e.~$\widehat{A}(TM)\cap [M]\in H_*(M,\partial M;\Q )$. Suppose that $(M,g)$ has positive scalar curvature and the bundle $\bundle{E}\to M$  has curvature $\| R^\bundle{E}\| _g<\frac{\min \mathrm{scal}_g}{2n(n-1)}$, then $\mathrm{ind}_\mathrm{APS}(\Di ^+)=0$ by the same arguments as in \cite[sec.~3]{APS,APS2}: If $\mathrm{ind}_\mathrm{APS}(\Di ^+)\neq 0$, then there is a nontrivial spinor $\psi \in \Gamma (\spinor \hat M\otimes \bundle{E})$ on the complete manifold $\hat M:=M\cup (-\infty ,0]\times \partial M$ with $\Di ^+\psi =0$ (if $\mathrm{ind}_{\mathrm{APS}}(\Di ^+)<0$ reverse orientation). Since $\psi $ decays exponentially at infinity, we conclude $\int _{\hat M}\left< \nabla ^*\nabla \psi ,\psi \right> =\int _{\hat M}\left< \nabla \psi ,\nabla \psi \right> $. Thus, the integrated version of the Lichnerowicz formula
\[
\Di ^2=\nabla ^*\nabla +\frac{1}{4}\mathrm{scal}_g+\frak{R}^\bundle{E}
\]
yields a contradiction to $\Di ^+\psi =0$ if $\min \mathrm{scal}_g>4\| \frak{R}^\bundle{E}\| _g\geq 2n(n-1)\| R^\bundle{E}\| _g$. Using the bundles $\bundle{E}$ and $\bundle{E}\oplus \C $ we conclude
\[
\int _M\widehat{A}_{n/4}(TM,\nabla ^{TM})=\frac{\eta +h}{2}\qquad \text{and}\qquad \left< \widetilde{\mathrm{ch}}(\bundle{E}),\widehat{A}(TM)\cap [M]\right> =0
\]
if $\| R^{(\bundle{E},\nabla )}\| _g< \frac{\min \mathrm{scal}_g}{2n(n-1)}$. Thus, if $(M,g)$ is an even dimensional Riemannian spin manifold of positive scalar curvature whose metric is a product near the boundary, then $\widehat{A}(TM)\cap [M]\in H_{2*}(M,\partial M;\Q )$ has finite K--area by the relative version of \cite[prop.~3.1]{List10}. Moreover, $(\partial M,g_{|})$ has positive scalar curvature which implies $h=\dim \ker \dirac _{\partial M}=0$. This completes the proof of theorem \ref{thm_aps_scalar} if $n=\dim M$ is even. If $n$ is odd, apply the even case to the manifold $M\times S^1$. Note that the constant in \cite[prop.~3.1]{List10} and the above arguments yield the rough estimate
\[
\ke (M_g^n,\partial M;\widehat{A}(TM)\cap [M])\leq \frac{n(n+1)^3}{2\cdot \min \, \mathrm{scal}_g} 
\]
if $g$ is a product near the boundary and $\mathrm{scal}_g>0$.
\begin{rem}
The concept of enlargeability introduced by Gromov and Lawson generalizes to compact manifolds with boundary as follows. Let $(S^n,g_0)$ be the standard sphere and $(M^n,g)$ be a compact manifold. Then $M$ is said to be \emph{compactly $\Lambda ^2$--enlargeable} if for all $\epsilon >0$, there is a finite Riemannian covering $(\tilde M,\tilde g)\to (M,g)$ which is trivial at the boundary and a smooth map $f:\tilde M\to S^n$ such that $P=f(\partial \tilde M)$ is a finite set of points in $S^n$, $f_*:H_n(\tilde M,\partial \tilde M)\to H_n(S^n,P)$ is nontrivial and $f^*g_0\leq \epsilon \cdot \tilde g$ holds on $\Lambda ^2T\tilde M$. The proof of \cite[prop.~4.1]{List10} generalizes easily to manifolds with boundary. In fact, if a spin manifold $M$ is compactly $\Lambda ^2$--enlargeable, then $\H _n(M,\partial M)\neq H_n(M,\partial M)$ implies that $M$ does not admit a metric of positive scalar curvature which is a product near the boundary.
\end{rem}
\section{An Uhlenbeck type result}
The theorem below is not vital for the main results in the introduction, however it is an interesting supplement to theorem \ref{thm_uhl}. The constant $C$ in theorem \ref{thm_uhl} can not be better than the constant
\[
\mlambda=\mlambda(g):=\sup _{0\neq \alpha \in \Omega ^1M\atop d^*\alpha =0}\frac{\| \alpha \| _g }{\|d\alpha \| _g } 
\] 
which is finite if $H^1(M;\R )=0$. In fact, in the abelian case $m=\mathrm{rk}(\bundle{E})=1$ theorem \ref{thm_uhl} holds for $C=\mlambda $ and arbitrary $\epsilon >0$. In the nonabelian case however, $\epsilon $ must be small and the best constant is $2\mlambda $. We notice that  the $\epsilon $ in the following theorem may be much smaller than the $\epsilon $ in theorem \ref{thm_uhl}.
\begin{thm}
\label{thm_uhlen}
Let $(M^n,g)$ be a simply connected closed Riemannian manifold, then there is a constant $\epsilon =\epsilon (M,g)>0$ with the following property: If $(\bundle{E},\nabla )\to M$ is a Hermitian vector bundle with curvature $\| R^\nabla \| _g<\epsilon $, then there is a unitary trivialization $\Psi :\bundle{E}\to M\times \C ^m$ such that $A:=\Psi \nabla \Psi ^{-1}-d\in \Omega ^1(M)\otimes \frak{u}(m)$ satisfies $d^*A=0$,
\[
\| A\| _g\leq 2\mlambda \cdot \| R^\nabla \| _g\qquad \text{and}\qquad \| dA\| _g\leq 2 \| R^\nabla \| _g.
\]
Moreover, $\Psi $ is unique up to multiplication by a unitary matrix.
\end{thm}
\begin{lem}
\label{lem62}
Let $(M,g)$ be a closed Riemannian manifold with $H^1(M;\R )=\{ 0\}$. If $A\in \Omega ^1(M)\otimes \frak{u}(m)$ satisfies $d^*A=0$, $\| A\| _g\leq \frac{1}{4\mlambda } $ and $\| R^{d+A}\| _g< \frac{1}{8\mlambda ^2}$, then
\[
\| A\| _{g}\leq 2\mlambda \cdot \| R^{d+A}\| _{g}\qquad \text{and}\qquad \| dA\| _{g}\leq 2\cdot \| R^{d+A}\| _{g}.
\]
\end{lem}
\begin{proof}
Let $|A(v)|_{op}$ be maximal for $v\in T_x M$, then there is some $S\in \mathrm{U}(m)$ such that $\tilde A:=S^{-1}AS:TM\to \frak{u}(m)$ is diagonal for $v$ with $|\tilde A(v)_{11}|_{op}\geq |\tilde A(v)_{ii}|_{op}$ for all $i$ (here $\tilde A(w):=S^{-1}A(w)S$ for all $w\in TM$, $\tilde A(v)\in \frak{u}(m)$ is diagonal). Hence, $\tilde A _{11}(w):=\tilde A(w)_{11}$ defines a 1--form with $d^*\tilde A_{11}=0$ which means
\[
\| A\| _{g}=\| \tilde A\| _{g} = \| \tilde A_{11}\| _{g}\leq \mlambda \cdot \| d\tilde A_{11}\| _{g}\leq \mlambda \cdot \| d\tilde A\| _{g}=\mlambda \cdot \| dA\| _{g}.
\]
Note $d(S^{-1}AS)=S^{-1}(dA)S$ and $d(\tilde A_{ij})=(d\tilde A)_{ij}$. This yields
\[
\| A\| _{g}\leq \mlambda \cdot \| dA\| _{g}= \mlambda \cdot \| R^{d+A}-A\wedge A\| _{g}\leq \mlambda \cdot \| R^{d+A}\| _{g}+2\mlambda \cdot \| A\| _{g}^2
\]
The polynomial $2\mlambda x^2-x+\mlambda \| R^{d+A}\| _{g} $ has two different zeros for $\| R^{d+A}\| _{g} <\frac{1}{8\mlambda ^2}$ with minimal point at $x=\frac{1}{4\mlambda }$. Hence, if $x\leq 2\mlambda x^2+\mlambda \| R^{d+A}\| _{g} $ and $x\leq \frac{1}{4\mlambda }$, $x$ has to be smaller than the first zero of the polynomial. In fact, $\| A\| _{g}\leq \frac{1}{4\mlambda }$ implies
\[
\| A\| _{g}\leq \frac{1}{4\mlambda }\left( 1-\sqrt{1-8\mlambda ^2\|R^{d+A}\| _{g}}\right) \leq 2\mlambda \cdot \| R^{d+A}\| _{g} 
\] 
\[
\| dA\| _{g}\leq \| R^{d+A}\| _{g}+2\| A\| _{g}^2\leq \| R^{d+A}\| _{g}+\frac{1}{2\mlambda }\| A\| _{g}\leq 2\| R^{d+A}\|_{g}
\]
here we use $-\sqrt{1-y}\leq -1+y$ for $y\in [0,1]$.
\end{proof}
\begin{rem}
\[
\frak{A}_\xi :=\left\{ A\in \Omega ^1(M)\otimes \frak{u}(m)\ |\ \ \| A\| _g< \xi ,\ \| R^{d+A}\| _g<\frac{1}{8\mlambda ^2}\right\} 
\]
is contractible and hence, connected for all $\xi \leq\frac{1}{4\mlambda }$: If $t\in (0,1]$ and $A\in \frak{A}_\xi $, then 
\[
\begin{split}
\| R^{d+tA}\| _g&=\| tR^{d+A}+t(t-1)A\wedge A\| _g\leq t\| R^{d+A}\| _g+2t(1-t)\| A\| ^2_g< \frac{2t-t^2}{8\mlambda ^2}
\end{split}
\]
and $\| R^d\| _g=0$ yield $\| R^{d+tA}\| _g<\frac{1}{8\mlambda ^2}$ for all $t\in [0,1]$.
\end{rem}

We follow the proofs of \cite[theorem 2.5]{Uhl2} respectively \cite[lemma 2.7]{Uhl1}. Choose the constant $\kappa <+\infty $ with $\| df\| _g\leq \kappa \cdot \| d^*df\| _\infty $ for all smooth $f:M\to \R $. Consider the smooth map
\[
\begin{split}
F:W^{k+1,2 }_0(M,\frak{u}(m))&\times W^{k,2}(M,T^*M\otimes \frak{u}(m))\to W^{k-1,2 }_0(M,\frak{u}(m))\ ,\\
 (u,A)&\mapsto d^*(e^{-u}de^u+e^{-u}Ae^{u})
\end{split}
\] 
for $k>\frac{n}{2}+1$ where $W^{k,2 }(M,.)$ is the usual Sobolev space with respect to the Killing form on $\frak{u}(m)$ and $u\in W_0^{k,2 }(M,\frak{u}(m))$ means additionally $\int u=0$. The linearization of $F$ at $(0,A)$ is given by 
\[
(u,B)\mapsto d^*(du-uA+Au+B )
\]
and the $L^2$--adjoint $(dF)^*$ of the operator $u\mapsto dF_{(0,A)}(u,0)$ satisfies
\[
(dF)^*\psi =d^*d\psi -\sum _{i=1}^n[A(e_i),d\psi (e_i)]. 
\]
where $e_1,\ldots, e_n\in TM$ denotes an orthonormal basis. Because $\int \psi =0$, the inequality
\[
\| (dF)^*\psi \| _{\infty }\geq \| d^*d\psi \| _{\infty }-2n\| d\psi \| _{g }\cdot \| A\| _{g}\geq \frac{1}{\mlambda '}\left( \frac{1}{\kappa }-2n\| A\| _{g}\right) \| \psi \| _{\infty }
\]
proves that $u\mapsto dF_{(0,A)}(u,0)$ is an isomorphism if $\frac{1}\kappa >2n \| A\| _g$. Here, the inequality $\| d^*d\psi \| _\infty \geq \frac{1}{\kappa }\| d\psi \| _g$ follows by the same argument as in the proof of the above lemma and $\| d\psi \| _g\geq \frac{1}{\mlambda'}\| \psi \| _\infty $ for a constant $\mlambda'>0$ is obvious if $\int \psi =0$. Note that $W^{k+1,2}_0\ni u\mapsto dF_{(0,A)}(u,0)\in W^{k-1,2}_0$ is a Fredholm operator with trivial index, because $dF_{(0,A)}(.,0)-d^*d$ is compact for smooth $A$ and $k>\frac{n}{2}+1$. Hence, the implicit function theorem provides a smooth map
\[
h:W^{k,2}(M,T^*M\otimes \frak{u}(m))\to W^{k+1,2 }_0(M,\frak{u}(m))\ ,\ A\mapsto h(A)
\]
with $F(h(A),A)=0$ if $\| A\| _g<\frac{1}{2n\kappa }$. If $A$ is smooth, $h(A)$ is smooth and the map
\[
A\mapsto \tilde A:=e^{-h(A)}(de^{h(A)}+Ae^{h(A)})
\]
depends smooth on $A$ with $d^*\tilde A=0$. Thus, assuming $\| R^{d+A}\| _g=\| R^{d+\tilde A}\| _g<\frac{1}{8\mlambda ^2}$, the value $\| \tilde A\| _g=\frac{1}{4\mlambda }$ is impossible by lemma \ref{lem62} which means $\| \tilde A\| _g<\frac{1}{4\mlambda }$ if $\| A\| _g<\xi :=\min \{ \frac{1}{2n\kappa },\frac{1}{4\mlambda }\} $. Here we use that the image of the map $\frak{A}_\xi \ni A\mapsto \tilde A$ must be connected and $A=0\mapsto \tilde A=0$. Now we  use theorem \ref{thm_uhl} to complete the proof. Let $C$ and $\epsilon $ be the constants from theorem \ref{thm_uhl} and set $\tilde \epsilon :=\min \left\{ \epsilon ,\frac{1}{8\mlambda ^2},\frac{\xi }{C }\right\} $. Suppose that $(\bundle{E},\nabla )$ is a Hermitian bundle on $M$ with curvature $\| R^\nabla \| _g<\tilde \epsilon $, then theorem \ref{thm_uhl} provides a trivialization $\Psi :\bundle{E}\to M\times \C ^m$ with $\| A\| _g\leq C\cdot \| R^\nabla \| _g<\xi  $ for $A=\Psi \nabla \Psi ^{-1}-d$. Consider the trivialization $\tilde \Psi =e^{-h(A)}\Psi $, then $\tilde A=\tilde \Psi \nabla \tilde \Psi ^{-1}-d$ is given by $\tilde A=e^{-h(A)}(de^{h(A)}+Ae^{h(A)})$. Hence, $\tilde A$ satisfies $d^*\tilde A=0$, $\| \tilde A\| _g<\frac{1}{4\mlambda }$ and $\| R^{d+\tilde A}\| _g=\| R^\nabla \| _g<\frac{1}{8\mlambda ^2}$. Lemma \ref{lem62} completes the proof of theorem \ref{thm_uhlen} with $\tilde \epsilon \sim \epsilon  $ and $\tilde \Psi \sim \Psi $. Uniqueness of $\tilde \Psi $ follows analogously.

\bibliographystyle{abbrv}
\bibliography{/home/benutzer/math/bib/bibliothek}

\end{document}